\documentclass[11pt,a4paper,reqno,dvipdfmx]{amsart}

\usepackage[truedimen,margin=25truemm]{geometry} 
\usepackage{amsmath,amssymb,amsthm,amscd}
\usepackage{extarrows,mathtools}

\title[Explicit description of certain 3-point K-theoretic Gromov-Witten invariants]
{Explicit description of certain 3-point K-theoretic Gromov-Witten invariants for flag manifolds}

\author[S.~Naito]{Satoshi Naito}
\address[Satoshi Naito]{Department of Mathematics, Institute of Science Tokyo, 
2-12-1 Oh-okayama, Meguro-ku, Tokyo 152-8551, Japan.}
\email{naito.s.ac@m.titech.ac.jp}

\author[D.~Sagaki]{Daisuke Sagaki}
\address[Daisuke Sagaki]{Department of Mathematics, 
Institute of Pure and Applied Sciences, University of Tsukuba, 
1-1-1 Tennodai, Tsukuba, Ibaraki 305-8571, Japan.}
\email{sagaki@math.tsukuba.ac.jp}

\keywords{quantum $K$-theory, divisor axiom, Gromov-Witten invariants, 
quantum Bruhat graph, Chevalley formula \newline
Mathematics Subject Classification 2020: 
Primary 14N35; Secondary 14M15, 14N15, 14N10, 05E14.}

\allowdisplaybreaks
\numberwithin{equation}{section}

\newcommand{\Fg}{\mathfrak{g}}

\newcommand{\Ft}{\mathfrak{t}}

\newcommand{\BZ}{\mathbb{Z}}
\newcommand{\BQ}{\mathbb{Q}}
\newcommand{\BR}{\mathbb{R}}
\newcommand{\BC}{\mathbb{C}}

\newcommand{\BP}{\mathbb{P}}

\newcommand{\CA}{\mathcal{A}}

\newcommand{\CO}{\mathcal{O}}
\newcommand{\CM}{\mathcal{M}}

\newcommand{\sQ}{\mathsf{Q}}
\newcommand{\sS}{\mathsf{S}}

\newcommand{\sC}{\mathsf{C}}
\newcommand{\st}{\mathsf{t}}

\newcommand{\vp}{\varphi}
\newcommand{\vpi}{\varpi}
\newcommand{\eps}{\epsilon}

\newcommand{\be}{\mathbf{e}}
\newcommand{\bk}{\mathbf{k}}
\newcommand{\bm}{\mathbf{m}}
\newcommand{\bp}{\mathbf{p}}
\newcommand{\bq}{\mathbf{q}}

\newcommand{\bB}{\mathbf{B}}
\newcommand{\bC}{\mathbf{C}}
\newcommand{\bI}{\mathbf{I}}
\newcommand{\bK}{\mathbf{K}}
\newcommand{\bM}{\mathbf{M}}
\newcommand{\bzero}{\mathbf{0}}

\DeclareMathOperator{\wt}{wt}
\DeclareMathOperator{\qwt}{qwt}
\DeclareMathOperator{\ed}{end}

\DeclareMathOperator{\ev}{ev}
\DeclareMathOperator{\gch}{gch}

\DeclareMathOperator{\dnn}{down}
\DeclareMathOperator{\codim}{codim}
\DeclareMathOperator{\Inv}{Inv}

\newcommand{\af}{\mathrm{af}}

\newcommand{\pt}{\mathrm{pt}}

\newcommand{\Hom}{\mathrm{Hom}}
\newcommand{\QBG}{\mathrm{QBG}}

\newcommand{\Lie}{\mathrm{Lie}}

\newcommand{\edge}[1]{ \xrightarrow{\hspace{2pt}#1\hspace{2pt}} }
\newcommand{\Qe}[1]{ \xrightarrow[\mathsf{q}]{\hspace{2pt}#1\hspace{2pt}} }
\newcommand{\Be}[1]{ \xrightarrow[\mathsf{B}]{\hspace{2pt}#1\hspace{2pt}} }

\newcommand{\lng}{w_{\circ}}

\newcommand{\ti}[1]{\widetilde{#1}}

\newcommand{\ol}[1]{\overline{#1}}

\newcommand{\pair}[2]{\langle #1,\,#2 \rangle}
\newcommand{\Kmet}[2]{(\!( #1,\,#2 )\!)}
\newcommand{\bra}[1]{[\![#1]\!]}
\newcommand{\pra}[1]{(\!(#1)\!)}

\newcommand{\twp}[2]{ \langle \CO^{#1}, \CO_{#2} \rangle }
\newcommand{\twpx}[2]{ \langle #1, #2 \rangle }

\newcommand{\thpx}[3]{ \langle #1, #2, #3 \rangle }

\newcommand{\QG}{\mathbf{Q}_{G}}
\newcommand{\QGr}{\mathbf{Q}_{G}^{\mathrm{rat}}}

\newcommand{\LQG}[1]{[\CO_{\QG}(#1)]}
\newcommand{\SQG}[1]{[\CO_{\QG(#1)}]}

\newcommand{\bfR}{\mathbf{R}}
\newcommand{\sprod}{\sideset{}{^\star}\prod}

\newcommand{\bchi}{\pmb{\chi}}
\newcommand{\Par}{\mathop{\rm Par}\nolimits}

\theoremstyle{plain}
\newtheorem{lem}{Lemma}[section]
\newtheorem{prop}[lem]{Proposition}
\newtheorem{thm}[lem]{Theorem}
\newtheorem{cor}[lem]{Corollary}

\newtheorem{ithm}{Theorem}

\theoremstyle{definition}
\newtheorem{dfn}[lem]{Definition}

\theoremstyle{remark}
\newtheorem{ex}[lem]{Example}
\newtheorem{rem}[lem]{Remark}
\newtheorem{claim}{Claim}[lem]

\newtheorem{irem}[ithm]{Remark}

\newcommand{\bqed}{\quad \hbox{\rule[-0.5pt]{3pt}{8pt}}}

\newenvironment{enu}{%
 \begin{enumerate}%
}{\end{enumerate}}

\begin{document}


%
\begin{abstract}
We give an explicit description, in terms of the quantum Bruhat graph, 
of the (torus-equivariant) 3-point, genus 0, $K$-theoretic Gromov-Witten invariants 
$\langle \CO(- \lambda), \CO^{w}, \CO_{u} \rangle_{d}$ for 
the (full) flag manifold $X = G/B$, where $\CO(- \lambda)$ denotes 
the class in the (torus-equivariant) $K$-theory ring $K_{T}(X)$ of $X$ of 
the line bundle $\CO_{X}(- \lambda) = G \times_{B} \BC_{\lambda}$ over $X = G/B$ 
associated to a weight $\lambda \in W \varpi_i$ lying 
in the Weyl group orbit of a minuscule fundamental weight $\varpi_i$, 
and $\CO_{u}$, $\CO^{w}$ are the Schubert and opposite Schubert classes 
in $K_{T}(X)$ for $u, w \in W$. 
This result can be thought of as a partial generalization of 
the quantum $K$-theoretic divisor axiom, which we obtained 
in our previous work; our proof utilizes a generalization of 
the Chevalley formula in the (torus-equivariant) quantum $K$-theory ring 
$QK_{T}(X)$ of $X$, which computes the quantum product 
with the line bundle class $\CO(- \lambda)$ associated to the weight $\lambda$ above. 
\end{abstract}

\maketitle

\section{Introduction.}

Let $G$ be a connected, simply-connected, simple (linear) algebraic group over $\BC$, 
with $B$ a Borel subgroup and $T$ a maximal torus contained in $B$; 
let $W = N_{G}(T)/T$ be the Weyl group of $G$ generated by 
the simple reflections $s_j$, $j \in I$. 
Let $X := G/B$ be the corresponding (full) flag manifold. 

For an effective degree $d \in H_2(X; \BZ)$ and $m\geq 0$, 
we let $\ol{\CM}_{0,m}(X,d)$ denote 
the Kontsevich moduli space of $m$-point stable maps to 
$X$ of genus zero and degree $d$ (see \cite{FP}, \cite{Tho}). 
This moduli space is non-empty when $d\neq 0$ or $m\geq 3$. 
In this case, it is equipped with ($T$-equivariant) evaluation maps 
$\ev_k:\ol{\CM}_{0,m}(X,d)\to X$ for $1 \leq k\leq m$; 
the map $\ev_k$ sends a stable map to the image of the $k$-th marked point 
in its domain. For classes 
$\gamma_k \in K_{T}(X)$, $1 \leq k \leq m$, 
the corresponding $m$-point ($T$-equivariant) 
$K$-theoretic Gromov-Witten (KGW) invariant is defined to be
\begin{equation*}
\langle \gamma_1, \gamma_2, \ldots , \gamma_m \rangle_d := 
\chi^{T}_{\ol{\CM}_{0,m}(X,d)}\left(\prod_{k=1}^m\ev_k^*\gamma_k \right) 
\in K_{T}(\pt), 
\end{equation*}
where $\chi^{T}_{\ol{\CM}_{0,m}(X,d)} : K_{T}(\ol{\CM}_{0,m}(X,d)) \to K_{T}(\pt)$ is 
the pushforward map along the structure morphism : $\ol{\CM}_{0,m}(X,d) \to \pt$; 
here, $K_{T}(\pt)$ is isomorphic to the representation ring $R(T)$ of $T$, 
which is also identified with the group algebra $\BZ[\Lambda]$ of 
the weight lattice $\Lambda := \sum_{i \in I} \BZ \varpi_i$, 
with $\varpi_i$ for $i \in I$ the $i$-th fundamental weight for $G$. 

Most studies in the ($T$-equivariant) quantum $K$-theory of 
homogeneous spaces have focused on $3$-point ($T$-equivariant) KGW invariants. 
These invariants govern the ($T$-equivariant) small quantum $K$-theory ring 
$QK_T(X) = K_T(X) \otimes_{K_T(\pt)} K_T(\pt)\bra{\sQ}$ of $X = G/B$, 
equipped with the quantum product $\star$, 
where $K_{T}(\pt)\bra{\sQ}$ is the formal power series ring 
in the Novikov variables $\sQ:=(\sQ_j \mid j \in I)$ 
with coefficients in $K_{T}(\pt)$; the ring $QK_{T}(X)$ was 
introduced by \cite{Giv} and \cite{Lee}, as a deformation of 
the ordinary ($T$-equivariant) $K$-theory ring $K_T(X)$. 

Let $\Delta$ be the root system of $G$, 
$\bigl\{ \alpha_j \bigr\}_{j \in I} \subset \Delta$ 
the set of simple roots, $\bigl\{ \alpha_j^{\vee} \bigr\}_{j \in I} \subset \Delta^{\vee}$ 
the set of simple coroots, where $\Delta^{\vee}$ is the coroot system of $G$; 
note that the homology group $H_2(X; \BZ)$ is 
identified with the coroot lattice $Q^{\vee} := 
\sum_{j \in I} \BZ \alpha_j^{\vee}$ of $G$, 
with $\alpha_j^{\vee}$ corresponding to 
the class $[X_{s_j}]$ of the Schubert curve $X_{s_j} \subset X$ for $j \in I$. 
Let $\theta \in \Delta$ denote the highest root of $\Delta$, 
with $\theta^{\vee} \in \Delta^{\vee}$ its coroot; 
also, let $\vartheta^{\vee} \in \Delta^{\vee}$ denote 
the highest coroot in the coroot system $\Delta^{\vee}$, 
which is not to be confused with the coroot $\theta^{\vee}$ 
of the highest root $\theta$. 

For $u, w \in W$, let $X_{u} \subset X$ and $X^{w} \subset X$ be 
the Schubert and opposite Schubert varieties, respectively, 
and let $\CO_u$ (resp., $\CO^w$) denote the class in $K_{T}(X)$ 
of the structure sheaf $\CO_{X_{u}}$ of $X_u$ (resp., $\CO_{X^w}$ of $X^w$). 
Also, for a weight $\nu \in \Lambda = \sum_{i \in I} \BZ \varpi_i$, 
let $\CO(\nu)$ denote the class $[\CO_{X}(\nu)]$ in $K_{T}(X)$ of 
the $G$-equivariant line bundle $\CO_{X}(\nu)$ over $X = G/B$ constructed as the quotient 
$G \times^{B} \BC_{- \nu}$ associated to 
the one-dimensional $B$-module $\BC_{-\nu}$ of $B$ of weight $- \nu$. 

In our previous paper \cite{LNSX}, 
we proved the following theorem (see \cite[Theorems~3.1 and 3.2]{LNSX}); 
here, we have used the well-known identity $\CO^{s_i} = 1 - \be^{- \varpi_{i}} \CO(- \varpi_{i})$ in $K_{T}(X)$ (and hence in $QK_{T}(X)$) for $i \in I$. 
%
%
\begin{ithm}\label{ithm:divisor}
Let $i \in I$, and $u, \, w \in W$. 
Then, for an effective degree 
$d = \sum_{j \in I} d_j \alpha_j^{\vee} \in Q^{\vee,+} := 
\sum_{j \in I} \BZ_{\geq 0} \alpha_j^{\vee}$ such that $d_i=0$, 
the following holds in $QK_T(X)${\rm :} 
\begin{equation}\label{eq:part1}
\langle \CO(- \varpi_i), \CO^{w}, \CO_{u} \rangle_{d} =  
\langle \CO(- \varpi_i) \cdot \CO^{w}, \CO_{u} \rangle_{d},
\end{equation}
where $\CO(- \varpi_i) \cdot \CO^{w}$ denotes the ordinary product in $K_T(X)$. 
Also, for $i \in I$ such that $\pair{\vpi_i}{\theta^{\vee}} = 1$ and 
an effective degree 
$d = \sum_{j \in I} d_j \alpha_j^{\vee} \in Q^{\vee,+}$ such that $d_i > 0$, 
the following holds in $QK_T(X)${\rm :} 
\begin{equation} \label{eq:part2}
\langle \CO(- \varpi_i), \CO^w, \CO_u \rangle_{d} =  0. 
\end{equation}
\end{ithm}

Moreover, for arbitrary $i \in I$, $u, w \in W$, and 
effective degree $d \in Q^{\vee,+}$, we gave an explicit description of 
the 3-point KGW invariant $\langle \CO(- \varpi_i), \CO^{w}, \CO_{u} \rangle_{d}$ 
in terms of the quantum Bruhat graph associated to $W$. 

In this paper, we consider a weight $\lambda \in \Lambda$ of 
the form $\lambda = x \varpi_i$ for some $x \in W$, where $\varpi_i$ is a minuscule fundamental weight, 
i.e., the fundamental weight corresponding to an index $i \in I$ 
such that $\pair{\vpi_i}{\vartheta^{\vee}} = 1$, 
where $\vartheta^{\vee} \in \Delta^{\vee}$ is the highest coroot; 
note that in this case, we have $\pair{\lambda}{\alpha^{\vee}} = 0,\,\pm 1$ 
for all coroots $\alpha^{\vee} \in \Delta^{\vee}$. 

Let $\QBG(W)$ be the quantum Bruhat graph associated to 
the Weyl group $W$ of $G$ (for details, see Definition~\ref{dfn:QBG}). 
For a directed path $\bm$ in $\QBG(W)$, let $\ell(\bm)$ be its length, 
$\ed(\bm) \in W$ its endpoint, and $\qwt(\bm) \in Q^{\vee,+}$ 
its quantum weight (i.e., the sum of the coroots of the labels 
of all quantum edges in $\bm$); for details, see Section~\ref{subsec:QBG}. 
Also, for $w \in W$, let $\bM_{w}$ denote a certain set of 
directed paths (whose labels are increasing with respect to a fixed reflection order) 
$\bm$ in $\QBG(W)$ starting at $w \in W$ such that 
$\bm$ is the concatenation of directed paths 
$\bm_{\beta}$ (the $\beta$-part) and $\bm_{\gamma}$ (the $\gamma$-part) 
in $\QBG(W)$; for details, see \eqref{eq:def-bMw}.

\begin{irem}
In the special case that $w = e$, 
we can show that any directed path $\bm$ 
in the set $\bM_{e}$ does not contain quantum edges, 
and hence $\qwt(\bm) = 0$ for all $\bm \in \bM_{e}$; 
see Corollary~\ref{cor:gc1}. 
This fact is essential in the proof of Theorem~\ref{ithm:gchevalley} below. 
\end{irem} 

For $u,\,w \in W$ and 
$d = \sum_{j \in I} d_j \alpha_j^{\vee} \in Q^{\vee,+}$, with $d_j \geq 0$ for $j \in I$, 
the 3-point KGW invariant $\langle \CO(- \lambda), \CO^{w}, \CO_{u} \rangle_{d}$ 
can be explicitly described in terms of $\QBG(W)$ as follows (see Theorem~\ref{thm:3pt-lam}); 
note that there do not appear any negative signs in this expression. 
%
%
\begin{ithm}\label{ithm:minuscule}
Let $\lambda = x \varpi_i \in \Lambda$, 
with $x \in W$, be a weight for which $\varpi_i$ is a minuscule fundamental weight, 
and let $u,\,w \in W$, $d = \sum_{j \in I} d_j \alpha_j^{\vee} \in Q^{\vee,+}$, 
with $d_j \geq 0$ for $j \in I$. Then, the following holds in $QK_T(X)${\rm :} 
\begin{equation}
\langle \CO(- \lambda), \CO^{w}, \CO_{u} \rangle_{d} = 
\sum_{ \bm \in \bfR_{w, u, d} } \be^{\ed(\bm_{\beta}) \lambda}
\end{equation}
for a certain explicitly described subset $\bfR_{w, u, d}$ of $\bM_w$. 
\end{ithm}

For the precise definition of 
the subset $\bfR_{w, u, d}$ of $\bM_w$, see \eqref{eq:def-bfR}.

The proof of Theorem~\ref{ithm:minuscule} utilizes 
the following partial generalization (see Corollary~\ref{cor:gc2}) of the Chevalley formula 
in $QK_T(X)$, obtained in \cite{LNS} (see also \cite{NOS}), 
though we assume an additional condition that 
$\pair{\varpi_i}{\vartheta^{\vee}} = 1$, 
where $\vartheta^{\vee}$ is the highest coroot 
in the coroot system $\Delta^{\vee}$. 
%
%
\begin{ithm}\label{ithm:gchevalley}
Let $\lambda = x \varpi_i \in P$, with $x \in W$, be a weight 
for which $\varpi_i$ is a minuscule fundamental weight, 
and let $w \in W$. Then, the following holds in $QK_T(X)${\rm :} 
\begin{equation}
\CO(- \lambda) \star \CO^w = 
\sum_{\bm \in \bM_{w}} (-1)^{\ell(\bm_{\gamma})} 
\be^{\ed(\bm_{\beta}) \lambda} \BQ^{\qwt(\bm)} \CO^{\ed(\bm)}.
\end{equation} 
\end{ithm}

As a consequence of Theorem~\ref{ithm:minuscule}, 
we obtain the following result (see Theorem~\ref{thm:3pt-lam2}), which generalizes the first half of 
Theorem~\ref{ithm:divisor} under a little stronger condition 
$\pair{\varpi_i}{\vartheta^{\vee}} = 1$ than 
the condition $\pair{\varpi_i}{\theta^{\vee}} = 1$ 
in Theorem~\ref{ithm:divisor}, where $\vartheta^{\vee} \in \Delta^{\vee}$ is 
the highest coroot and $\theta^{\vee}$ is 
the coroot of the highest root $\theta \in \Delta$. 
%
%
\begin{ithm}\label{ithm:gdivisor}
With the same notation as above, assume that $d_j = 0$ for all $j \in I$ such that $\pair{\lambda}{\alpha_j^{\vee}} \not= 0$, or equivalently, assume that $\pair{\lambda}{d} = 0$. 
Then, the following holds in $QK_T(X)${\rm :}
\begin{equation}
\langle \CO(- \lambda), \CO^{w}, \CO_{u} \rangle_{d} =  
\langle \CO(- \lambda) \cdot \CO^{w}, \CO_{u} \rangle_{d},
\end{equation}
where $\CO(- \lambda) \cdot \CO^{w}$ denotes the ordinary product in $K_{T}(X)$. 
\end{ithm}

This paper is organized as follows. 
In Section~\ref{sec:prelim}, 
we fix the notation for root systems, recall the definition of 
the quantum Bruhat graph, and also explain basic facts about the equivariant $K$-group of 
semi-infinite flag manifolds and quantum $K$-theory of ordinary flag manifolds. 
In addition, we review the general Chevalley formula in the equivariant $K$-group 
of semi-infinite flag manifolds, which plays an important role 
in this paper. In Section~\ref{sec:main}, 
we we state and prove our main results given above. 
In Appendix~\ref{sec:example}, we give some examples of Theorems~\ref{ithm:minuscule} and \ref{ithm:gchevalley}. 

\medskip
\paragraph{\bf Acknowledgments.}
We would like to express our sincere thanks to 
Cristian Lenart for his valuable collaboration on related projects.
S.N. was partly supported by JSPS Grant-in-Aid for Scientific Research (C) 26K06759. 
D.S. was partly supported by JSPS Grant-in-Aid for Scientific Research (C) 23K03045.

%
\section{Preliminaries.}
\label{sec:prelim}
%
%
\subsection{Notation for root systems.} \label{subsec:alggrp}
Let $G$ be a connected, simply-connected, simple (linear) algebraic group over $\BC$, 
with $T$ a maximal torus of $G$, 
$B=TN$ a Borel subgroup of $G$, $N$ the unipotent radical of $B$.
Let $B^{-} \subset G$ be the opposite Borel subgroup to $B$, i.e., 
the unique Borel subgroup such that $B \cap B^- = T$. 
We set $\Fg := \Lie(G)$ and $\Ft := \Lie(T)$; 
$\Fg$ is a finite-dimensional simple Lie algebra over $\BC$, and 
$\Ft$ is a Cartan subalgebra of $\Fg$. 
We denote by $\pair{\cdot\,}{\cdot} : \Ft^{\ast} \times \Ft \rightarrow \BC$ 
the canonical pairing, where $\Ft^{\ast} \coloneqq   \Hom_{\BC}(\Ft, \BC)$. 
Let $\Delta \subset \Ft^{\ast}$ be the root system of $\Fg$, 
$\Delta^{+} \subset \Delta$ the set of positive roots, 
and $\{ \alpha_{j} \}_{j \in I} \subset \Delta^{+}$ the set of simple roots; 
we denote by $\theta \in \Delta^+$ the highest root of $\Delta$, 
and set $\rho := (1/2) \sum_{\alpha \in \Delta^{+}} \alpha$. 
The root lattice $Q$ and the coroot lattice $Q^{\vee}$ of $\Fg$ are 
$Q := \sum_{j \in I} \BZ \alpha_{j}$ and $Q^{\vee} := 
\sum_{j \in I} \BZ \alpha_{j}^{\vee}$, respectively. 
We set $Q^{\vee,+} \coloneqq  \sum_{j \in I} \BZ_{\ge 0} \alpha_{j}^{\vee}$; 
for $\xi,\zeta \in Q^{\vee}$, we write $\xi \ge \zeta$ if $\xi-\zeta \in Q^{\vee,+}$. 
For $i \in I$, the weight $\vpi_{i} \in \Ft^{\ast}$ 
given by $\pair{\vpi_{i}}{\alpha_{j}^{\vee}} = \delta_{i,j}$ for all $j \in I$, 
with $\delta_{i,j}$ the Kronecker delta, 
is called the $i$-th fundamental weight for $G$. 
Then the (integral) weight lattice $\Lambda$ of $G$ is 
$\Lambda \coloneqq \sum_{i \in I} \BZ \vpi_{i}$; 
we set $\Lambda^{+} := \sum_{i \in I} \BZ_{\geq 0} \varpi_i$ and 
$\Lambda^{++} := \sum_{i \in I} \BZ_{\geq 1} \varpi_i$. 
Also, we denote by $\BZ[\Lambda]$ the group algebra of $\Lambda$, that is, 
the associative algebra generated by formal elements $\be^{\nu}$, $\nu \in \Lambda$, 
where the product is defined by $\be^{\mu} \be^{\nu}:=
\be^{\mu + \nu}$ for $\mu, \nu \in \Lambda$. 

A reflection $s_{\alpha} \in GL(\Ft^{\ast})$, $\alpha \in \Delta$, 
is defined by: $s_{\alpha} \mu = \mu - \pair{\mu}{\alpha^{\vee}} \alpha$ 
for $\mu \in \Ft^{\ast}$, where $\alpha^{\vee} \in \Ft$ is the coroot of $\alpha \in \Delta$; 
we write $s_{j} \coloneqq   s_{\alpha_{j}}$ for $j \in I$. 
Then the (finite) Weyl group $W$ of $\Fg$ is defined to be the subgroup of $GL(\Ft^{\ast})$ 
generated by $\{ s_{j} \}_{j \in I}$, that is, $W \coloneqq   \langle s_{j} \mid j \in I \rangle$. 
For $w \in W$, there exist $j_{1}, \ldots, j_{r} \in I$ such that $w = s_{j_{1}} \cdots s_{j_{r}}$. 
If $r$ is minimal, then the product $s_{j_{1}} \cdots s_{j_{r}}$ is called a reduced expression for $w$, 
and $r$ is called the length of $w$; we denote by $\ell(w)$ the length of $w$. 
Denote by $\lng$ the longest element of $W$. 
Also, let $W_{\af} \cong W \ltimes Q^{\vee}$ be the affine Weyl group, 
where $t_{\xi} \in W_{\af}$ denotes the translation element for $\xi \in Q^{\vee}$; 
we set $W_{\af}^{\geq 0} := 
\bigl\{w t_{\xi} \in W_{\af} \mid w \in W,\,\xi \in Q^{\vee,+} \bigr\}$. 

%
\subsection{The quantum Bruhat graph.}
\label{subsec:QBG}

\begin{dfn} \label{dfn:QBG} 
The quantum Bruhat graph on $W$, 
denoted by $\QBG(W)$, is the $\Delta^{+}$-labeled
directed graph whose vertices are the elements of $W$ and 
whose edges are of the following form: 
$v \edge{\alpha} w$, with $v, w \in W$ and $\alpha \in \Delta^{+}$, 
such that $w = v s_{\alpha}$ and either of the following holds: 
(B) $\ell(w) = \ell (v) + 1$; 
(Q) $\ell(w) = \ell (v) + 1 - 2 \pair{\rho}{\alpha^{\vee}}$.
An edge satisfying (B) (resp., (Q)) is called a Bruhat edge (resp., quantum edge). 
\end{dfn}

Let 
\begin{equation} \label{eq:dp}
\bp: w_{0} \edge{\beta_{1}} w_{1} \edge{\beta_{2}} \cdots \edge{\beta_{r}} w_{r}
\end{equation}
be a directed path in the quantum Bruhat graph $\QBG(W)$. 
We set $\ed(\bp)\coloneqq  w_{r}$ and $\ell(\bp)=r$. 
A directed path $\bp$ is called the trivial (resp., non-trivial) one 
if $\ell(\bp)=0$ (resp., $\ell(\bp) > 0$). 
For $\bp$ of the form \eqref{eq:dp}, we set
\begin{equation} \label{eq:def-qwt}
\qwt(\bp) := \sum_{ \begin{subarray}{c} 1 \le k \le r \\ 
\text{$w_{k-1} \edge{\beta_k} w_{k}$ is} \\
\text{a quantum edge} \end{subarray} } \beta_{k}^{\vee} \in Q^{\vee,+}. 
\end{equation}
Let $v,w \in W$, and let $\bp$ be a shortest directed path from $v$ to $w$ in $\QBG(W)$. 
We set $\qwt(v \Rightarrow w)\coloneqq  \qwt(\bp)$; 
we know from \cite[Lemma~1]{Pos} (see also \cite[Proposition~8.1]{LNSSS1}) that 
$\qwt(v \Rightarrow w)$ does not depend on the choice of a shortest directed path $\bp$. 

Let $\lhd$ be a reflection (or convex) order on $\Delta^{+}$; 
see, e.g., \cite[\S 2.2]{KoNS}.
A directed path $\bp$ of the form \eqref{eq:dp} is said to be label-increasing 
(resp., label-decreasing) with respect to $\lhd$ 
if $\beta_{1} \lhd \cdots \lhd \beta_{r}$ (resp., $\beta_{1} \rhd \cdots \rhd \beta_{r}$). 
%
%
\begin{thm}[{see, for example, \cite[Theorem~7.4]{LNSSS1}}] \label{thm:label}
Let $v,w \in W$.
\begin{enu}
\item There exists a unique label-increasing directed path from $v$ to $w$ 
in the quantum Bruhat graph $\QBG(W)$. Moreover, it is a shortest directed path 
from $v$ to $w$. 

\item There exists a unique label-decreasing directed path from $v$ to $w$ 
in the quantum Bruhat graph $\QBG(W)$. Moreover, it is a shortest directed path 
from $v$ to $w$. 
\end{enu}
\end{thm}

%
\subsection{Equivariant $K$-groups of semi-infinite flag manifolds.}
\label{subsec:equivgroup}

First, let us briefly recall from \cite[\S2.3]{Kat1} 
the definitions of semi-infinite flag manifolds and semi-infinite Schubert varieties. 
Let $\QGr$ denote the semi-infinite flag manifold associated to $G$, 
which is a (reduced) ind-scheme of ind-infinite type whose set of 
$\BC$-valued points is $G(\BC\pra{z})/(T(\BC) \cdot N(\BC\pra{z}))$ 
(see also \cite{KaNS, Kat2}); note that $\QGr$ can be thought of 
as an inductive limit of copies of the (reduced) closed subscheme 
$\QG \subset \prod_{i \in I} \BP(L(\varpi_{i}) \otimes_{\BC} \BC\bra{z})$ 
of infinite type, introduced in \cite[\S4.1]{FM}, 
where $L(\varpi_{i})$ is the irreducible highest weight $G$-module 
of highest weight $\varpi_{i}$. 
One has the semi-infinite Schubert (sub)variety 
$\QG(y) \subset \QGr$ associated to each element $y$ of 
the affine Weyl group $W_{\af} \cong W \ltimes Q^{\vee}$; 
note that $\QG(y)$ is, by definition, the closure of 
the orbit under the Iwahori subgroup $\bI \subset G(\BC\bra{z})$, 
which is the pre-image of $B$ under the evaluation map $G(\BC\bra{z}) \rightarrow G$ at $z = 0$, 
through the ($T \times \BC^{*}$)-fixed point labeled by 
$y \in W_{\af}$ in exactly the same way as in 
\cite[\S4.2]{KaNS} and \cite[\S2.3]{Orr}. 
Also, note that $\QG(y)$ is contained in 
$\QG(e) = \QG$ for all $y \in W_{\af}^{\geq 0} = 
\bigl\{w t_{\xi} \in W_{\af} \mid w \in W,\,\xi \in Q^{\vee,+} \bigr\}$; 
we also call $\QG$ the semi-infinite flag manifold. 
For each weight $\nu = \sum_{i \in I} m_{i} \varpi_{i} \in P$ 
with $m_{i} \in \BZ$, one has a $(G(\BC\bra{z}) \rtimes \BC^{*})$-equivariant 
line bundle $\CO_{\QG}(\nu)$ over $\QG$, which is given by the restriction of 
the line bundle $\boxtimes_{i \in I} \CO(m_{i})$ over 
$\prod_{i \in I} \BP(L(\varpi_{i}) \otimes_{\BC} \BC\bra{z})$; 
we warn the reader that the convention for these line bundles is 
the same as that of \cite{KaNS}, but differs from that of \cite{Kat2} 
by the twist coming from the involution $- \lng$. 

Next, mainly following \cite[\S1.5]{Kat2}, 
we define some variants of equivariant $K$-group of 
the semi-infinite flag manifold $\QG$. 
The equivariant $K$-group $\ti{K}^{\prime}(\QG)$ of $\QG$ is, 
by definition, the $\BZ[q, q^{-1}][\Lambda]$-module consisting of 
all those formal infinite linear combinations 
$\sum_{y \in W_{\af}^{\geq 0}} a_{y} \, [\CO_{\QG(y)}]$ of 
the (semi-infinite Schubert) classes $[\CO_{\QG(y)}]$, $y \in W_{\af}$, 
of the structure sheaves $\CO_{\QG(y)}$ of $\QG(y)$ 
with coefficients $a_{y} \in \BZ[q, q^{-1}][\Lambda]$ for 
which the following holds for all regular dominant weights 
$\lambda \in \Lambda^{++}$: 
\begin{equation*}
\sum_{y \in W_{\af}^{\geq 0}} |a_{y}| \, 
\gch H^{0}(\QG, \CO_{\QG(y)} \otimes \CO_{\QG}(\lambda)) \in \BZ[\Lambda]\pra{q^{-1}}, 
\end{equation*}
where $\gch$ denotes the character of an $(T \times \BC^{*})$-diagonalizable module 
(see \cite[Corollary~4.31]{KaNS} for more details), 
and the absolute values $|a_{y}| \in \BZ[q, q^{-1}][\Lambda]$ 
for $y \in W_{\af}^{\geq 0}$ are taken coefficientwise; 
note that $\BZ[\Lambda] = \bigoplus_{\nu \in \Lambda} \BZ\be^{\nu} \cong R(T)$, 
and $q \in R(\BC^{*})$ corresponds to the loop rotation action of $\BC^{*}$. 
Then, we know from \cite[Theorem~1.25]{Kat2} that $\ti{K}^{\prime}(\QG)$ is 
stable under the tensor product $\bullet \otimes [\CO_{\QG}(\nu)]$ 
for all $\nu \in \Lambda$, where $\bullet$ is an arbitrary element of 
$\ti{K}^{\prime}(\QG)$. In particular, we have 
$[\CO_{\QG(y)}(\nu)] := [\CO_{\QG(y)}] \otimes [\CO_{\QG}(\nu)] \in \ti{K}^{\prime}(\QG)$ 
for all $y \in W_{\af}^{\geq 0}$ and $\nu \in \Lambda^{+}$; 
note that $[\CO_{\QG}(\nu)] = [\CO_{\QG(e)}] \otimes [\CO_{\QG}(\nu)] 
\in \ti{K}^{\prime}(\QG)$ for all $\nu \in \Lambda$. 
Also, we define $K_{T \times \BC^{*}}(\QG)$ to be the $\BZ[q, q^{-1}][\Lambda]$-submodule of 
$\ti{K}^{\prime}(\QG)$ consisting of 
all those convergent (in the sense of \cite[Proposition~5.11]{KaNS}) 
formal infinite linear combinations of the semi-infinite Schubert classes $[\CO_{\QG(y)}]$, 
$y \in W_{\af}^{\geq 0}$, with coefficients $a_{y} \in \BZ[q, q^{-1}][\Lambda]$; 
here convergence holds if the sum $\sum_{y \in W_{\af}^{\geq 0}} \vert a_{y} \vert$ 
of the absolute values $|a_{y}|$ lies in $\BZ[\Lambda]\pra{ q^{-1} }$. 
It is easy to deduce from \cite[Corollary~4.31]{KaNS} that 
$K_{T \times \BC^{*}}(\QG)$ is indeed a $\BZ[q, q^{-1}][\Lambda]$-submodule of 
$\ti{K}^{\prime}(\QG)$. In addition, it follows from \cite[Theorem~5.16]{KoLN} 
and (the proof of) \cite[Corollary~5.12]{KaNS} that 
the twisted semi-infinite Schubert class 
$[\CO_{\QG(y)}(\nu)] := [\CO_{\QG(y)}] \otimes [\CO_{\QG}(\nu)]$ 
lies in $K_{T \times \BC^{\ast}}(\QG) \subset \widetilde{K}^{\prime}(\QG)$ 
for all $y \in W_{\af}^{\geq 0}$ and $\nu \in \Lambda$. 

Now, following \cite[\S 1.5]{Kat2}, 
we define the $T$-equivariant $K$-group $K_{T}(\QG)$ of $\QG$ 
to be the specialization (of coefficients) at $q^{\pm 1} = 1$ of 
$\ti{K}^{\prime}(\QG)$ (or equivalently, $K_{T \times \BC^{\ast}}(\QG)$). 
Then, it turns out (cf. \cite[Lemma~1.22]{Kat2}) 
that $K_{T}(\QG)$ is just the $\BZ[\Lambda]$-module 
$\prod_{y \in W_{\af}^{\geq 0}} \BZ[\Lambda][\CO_{\QG(y)}]$ (direct product), 
which consists of all infinite linear combinations of 
the semi-infinite Schubert classes $[\CO_{\QG(y)}]$, $y \in W_{\af}^{\geq 0}$, 
with coefficients in $\BZ[\Lambda]$. 
Note that the semi-infinite Schubert classes $[\CO_{\QG(y)}]$, $y \in W_{\af}^{\geq 0}$, 
form a topological basis of $K_{T}(\QG)$; 
that is, an arbitrary element of $K_{T}(\QG)$ can be uniquely written 
as an infinite linear combination of the $[\CO_{\QG(y)}]$, $y \in W_{\af}^{\geq 0}$, 
with coefficients in $\BZ[\Lambda]$. Also, for each $\nu \in \Lambda$, 
a $\BZ[\Lambda]$-linear endomorphism $\bullet \otimes [\CO_{\QG}(\nu)]$ of $K_{T}(\QG)$ 
is induced from the $\BZ[q, q^{-1}][\Lambda]$-linear endomorphism 
$\bullet \otimes [\CO_{\QG}(\nu)]$ of $K_{T \times \BC^{\ast}}(\QG)$ 
by the specialization at $q^{\pm 1} = 1$ (see \cite[Theorem~1.26]{Kat2}); 
note that $[\CO_{\QG}(\nu)] = [\CO_{\QG(e)}] \otimes [\CO_{\QG}(\nu)] \in K_{T}(\QG)$ 
for $\nu \in \Lambda$. 
In addition, for $\xi \in Q^{\vee,+}$, a $\BZ[\Lambda]$-linear endomorphism 
$\st_{\xi}$ of $K_{T}(\QG)$, given by $\st_{\xi}[\CO_{\QG(y)}] := 
[\CO_{\QG(y t_{\xi})}]$ for $y \in W_{\af}^{\geq 0}$, 
is induced by the right action of $Q^{\vee}$ on $\QGr$ 
(see \cite[Eq.~(1.20)]{Kat2}). We know from \cite[Theorem~1.26]{Kat2} that 
\begin{equation}\label{eq:shift}
(\st_{\xi}[\CO_{\QG(y)}]) \otimes [\CO_{\QG}(\nu)] = 
\st_{\xi}([\CO_{\QG(y)}] \otimes [\CO_{\QG}(\nu)])
\end{equation}
for $y \in W_{\af}^{\geq 0}$ and $\nu \in \Lambda$. 
Since the semi-infinite Schubert classes $[\CO_{\QG(y)}]$, $y \in W_{\af}^{\geq 0}$, 
form a topological basis of $K_{T}(\QG)$ over $\BZ[\Lambda]$ in the sense above, 
it follows that
\begin{equation}
(\st_{\xi} \, \bullet) \otimes [\CO_{\QG}(\nu)] = 
\st_{\xi}(\bullet \otimes [\CO_{\QG}(\nu)])
\end{equation}
for an arbitrary element $\bullet \in K_{T}(\QG)$ and 
$\xi \in Q^{\vee,+}$, $\nu \in \Lambda$.

%
\subsection{General Chevalley formula.}
\label{subsec:genChev}

We review the general Chevalley formula for 
$K_{T}(\QG)$, following \cite{LNS} (cf. \cite[Theorem~5.16]{KoLN}). 
Let $\Ft^{\ast}_{\BR} :=\BR \otimes_{\BZ} \Lambda$ be a real form of $\Ft^{\ast}$, 
and set 
\begin{equation*}
	H_{\beta,l}:=\bigl\{ \mu \in \Ft^{\ast}_{\BR} \mid \pair{\mu}{\beta^\vee} =l\bigr\} 
\end{equation*}
for $\beta \in \Delta$ and $l \in \BZ$; we denote by $s_{\beta,l}$ the affine reflection in the affine hyperplane $H_{\beta,l}$.
The affine hyperplanes $H_{\beta,l}$, $\beta \in \Delta$, $l \in \BZ$, divide 
the real vector space $\Ft^{\ast}_{\BR}$ into open regions, called alcoves; 
the fundamental alcove is defined as
\begin{equation*}
	A_{\circ} := \bigl\{ \mu \in \Ft_{\BR}^{\ast} \mid 
	0 < \pair{\mu}{\alpha^\vee} < 1 \quad \text{for all $\alpha \in \Delta^{+}$} \bigr\}.
\end{equation*}
We say that two alcoves are adjacent if they are distinct and have a common wall. Given a pair
of adjacent alcoves $A$ and $B$, we write $A \edge{\beta} B$  for $\beta \in \Delta$ if the
common wall is orthogonal to $\beta$ and $\beta$ points in the direction from $A$ to $B$.

\begin{dfn}[\cite{LP}]
	An alcove path is a sequence of alcoves $(A_0, A_1, \ldots, A_m)$ such that
	$A_{j-1}$ and $A_j$ are adjacent for $j=1,\ldots,m$. 
	We say that $(A_0, A_1, \ldots, A_m)$ is reduced 
	if it has minimal length among all alcove paths from $A_0$ to $A_m$.
\end{dfn}

Let $\lambda \in \Lambda$, and let $A_{\lambda}=A_{\circ}+\lambda$ be 
the translation of the fundamental alcove $A_{\circ}$ by $\lambda$. 
	
\begin{dfn}[\cite{LP}] \label{dfn:lch}
    Let $\lambda \in \Lambda$. 
	A sequence $(\beta_1, \beta_2, \dots, \beta_m)$ of roots is called 
	a reduced $\lambda$-chain (of roots) if 
	\begin{equation}
		A_{\circ}=A_{0} \edge{-\beta_1}  A_1
		\edge{-\beta_2} \cdots 
		\edge{-\beta_m}  A_m=A_{-\lambda}
	\end{equation}
is a reduced alcove path.
\end{dfn}

A reduced alcove path $(A_0=A_{\circ},A_1,\ldots,A_m=A_{-\lambda})$ can be identified 
with the corresponding total order on the hyperplanes, to be called $\lambda$-hyperplanes, 
which separate $A_\circ$ from $A_{-\lambda}$. This total order is given by the sequence 
$H_{\beta_j,-l_j}$ for $j=1,\ldots,m$, where $H_{\beta_j,-l_j}$ contains the common wall of 
$A_{j-1}$ and $A_j$. Note that $\pair{\lambda}{\beta_j^\vee} \ge 0$, and 
that the integers $l_j$, called heights, have the following ranges:
\begin{equation*}
\begin{split}
& 0 \le l_j \le \pair{\lambda}{\beta_j^\vee}-1 \quad \text{if} \quad \beta_j \in \Delta^+, \\
& 1 \le l_j \le \pair{\lambda}{\beta_j^\vee} \quad \text{if} \quad \beta_j \in \Delta^- = - \Delta^+. 
\end{split}
\end{equation*}
Note also that a reduced $\lambda$-chain $(\beta_1, \ldots, \beta_m)$ determines 
the corresponding reduced alcove path, so we can identify them as well. 

Let $\lambda \in \Lambda$, and fix a reduced $\lambda$-chain, 
which we denote by $\Gamma(\lambda)=(\beta_1,\,\ldots,\,\beta_m)$. 
Let $w \in W$. 

\begin{dfn}[\cite{LL}] \label{dfn:admissible}
	A subset 
	$A=\left\{ j_1 < j_2 < \cdots < j_s \right\}$ of $[m]=\{1,\ldots,m\}$ (possibly empty)
 	is called a $w$-admissible subset if
	the following directed path:
	\begin{equation} \label{eqn:admissible}
	\begin{split}
	& \Pi(w,A): w \edge{|\beta_{j_1}|} w s_{\beta_{j_1}}
	\edge{|\beta_{j_2}|}  ws_{\beta_{j_1}}s_{\beta_{j_2}}
	\edge{|\beta_{j_3}|}  \cdots \\
	& \hspace{30mm} \cdots 
	\edge{|\beta_{j_s}|}  ws_{\beta_{j_1}}s_{\beta_{j_2}} \cdots s_{\beta_{j_s}}=:\ed(A) 
	\end{split}
	\end{equation}
	lies in the quantum Bruhat graph $\QBG(W)$. 
 	Let $\CA(w,\Gamma(\lambda))$ denote the collection of all $w$-admissible subsets of $[m] = \{1, \ldots, m\}$.
\end{dfn}

Let $A=\{ j_1 < \cdots < j_s\} \in \CA(w,\Gamma(\lambda))$. 
The weight of $A$ is defined by 
	\begin{equation} \label{eq:wta}
	\wt(A):=-w s_{\beta_{j_1},-l_{j_1}} \cdots s_{\beta_{j_s},-l_{j_s}}(-\lambda).
	\end{equation}
Also, we set 
\begin{equation} \label{eq:nA}
n(A):=\# \{ \beta_{j_1},\ldots,\beta_{j_s} \} \cap \Delta^{-}, 
\end{equation}
\begin{equation} \label{eq:A-}
A^{-}:=\left\{j_i \in A \ \biggm| \ 
\begin{array}{l}
\text{%
$ws_{\beta_{j_1}} \cdots s_{\beta_{j_{i-1}}} \edge{|\beta_i|}
 ws_{\beta_{j_1}} \cdots s_{\beta_{j_{i-1}}}s_{\beta_{j_{i}}}$} \\[2mm]
\text{is a quantum edge}
\end{array} \right\}, 
\end{equation}
\begin{equation} \label{def:height}
\dnn(A):=\sum_{j\in A^-}|\beta_j|^\vee\in Q^{\vee,+}.
\end{equation}

Write $\lambda \in \Lambda$ as $\lambda=\sum_{i\in I}\lambda_i\vpi_i \in \Lambda$, 
with $\lambda_{i} \in \BZ$ for $i \in I$. Following \cite[Section 4.1]{LNS}, 
let $\ol{\Par(\lambda)}$ denote the set of $I$-tuples of partitions 
$\bchi=(\chi^{(i)})_{i\in I}$ such that $\chi^{(i)}$ is a partition of 
length at most $\max(\lambda_i,0)$; in \cite[Section 2.5]{NSZ}, for $\lambda \in P^{+}$, 
we introduced the set $\Par(\lambda)$ of $I$-tuples of partitions $\bchi=(\chi^{(i)})_{i\in I}$ 
such that $\chi^{(i)}$ is a partition of length less than $\lambda_i$ 
(which we do not use in this paper).
For $\bchi = (\chi^{(i)})_{i \in I} \in \ol{\Par(\lambda)}$, 
we set $\iota(\bchi) := \sum_{i \in I} \chi^{(i)}_1 \alpha_i^{\vee} \in Q^{\vee,+}$, 
with $\chi^{(i)}_1$ the first part of the partition $\chi^{(i)}$ for each $i \in I$.

By specializing at $q = 1$ in \cite[Theorem~33]{LNS} (cf. \cite[Theorem~5.16]{KoLN}), 
we obtain the following general Chevalley formula for $K_{T}(\QG)$. 
%
%
\begin{thm} \label{thm:genchev} 
Let $\lambda=\sum_{i\in I}\lambda_i\vpi_i \in \Lambda$ be an arbitrary weight, 
$\Gamma(\lambda)$ an arbitrary reduced $\lambda$-chain, 
and $x=wt_{\xi}\in W_{\af}^{\ge 0}$. 
Then, the following equality holds in $K_{T}(\QG)${\rm :}
\begin{align}
& \LQG{-\lng \lambda} \otimes \SQG{x} \nonumber \\[3mm]
& \quad = \sum_{A \in \CA(w,\Gamma(\lambda))}
  \sum_{ \bchi \in \ol{\Par(\lambda)} }
  (-1)^{n(A)} \be^{\wt(A)}
  \SQG{ \ed(A)t_{\xi+\dnn(A)+\iota(\bchi)} }. \label{eq:gchev-org}
\end{align}
\end{thm}

%
\subsection{$K$-theoretic Gromov-Witten invariants for $G/B$.}
\label{subsec:KGW}

For any projective $T$-variety $Y$, we denote by $K_T(Y)$ 
the Grothendieck group of $T$-equivariant algebraic vector bundles over $Y$. 
This ring is an algebra over $K_T(\pt)=R(T)$, 
the representation ring of $T$, which is identified with 
the group algebra $\BZ[\Lambda]$ of $\Lambda$. 
Let $\chi_Y: K_T(Y) \rightarrow K_T(\pt)$ be the pushforward map 
along the structure morphism $Y \rightarrow \pt$.

Recall that the Weyl group $W= \langle s_j \mid j \in I \rangle$ of $G$ 
can be identified with $N_G(T)/T$, 
where $N_G(T)$ is the normalizer of $T$ in $G$. 
Let $X:=G/B$ be the (full) flag manifold. 
Each Weyl group element $w \in W$ defines 
the Schubert varieties $X_w = \ol{BwB/B}$ and $X^w=\ol{B^{-}wB/B}$ in $X$ 
with $\dim X_w = \codim X^w = \ell(w)$.
The equivariant $K$-theory ring $K_T(X)$ of 
the flag manifold $X=G/B$ has two $K_T(\pt)$-bases 
$\bigl\{ \CO_w \mid w \in W \bigr\}$ and 
$\bigl\{ \CO^w \mid w \in W \bigr\}$, 
where $\CO_w=[\CO_{X_w}]$ and $\CO^w=[\CO_{X^w}]$ are the 
Schubert classes defined by the structure sheaves of 
the Schubert varieties $X_w$ and $X^w$, respectively.

The homology group $H_{2}(X; \BZ)$ can be identified 
with $Q^{\vee}$, with $\alpha_j^\vee$ corresponding to 
the class $[X_{s_j}]$ of the Schubert curve $X_{s_j} \subset X$ for $j \in I$. 
For an effective degree $d \in Q^{\vee,+}$ and $m \geq 0$, 
let $\ol{\CM}_{0,m}(X,d)$ denote 
the Kontsevich moduli space of $m$-pointed stable maps to $X$ of 
genus zero and degree $d$ (see \cite{FP}, \cite{Tho}). 
This moduli space is non-empty when $d\neq 0$ or $m\geq 3$. 
In this case, it is equipped with ($T$-equivariant) evaluation maps 
$\ev_k:\ol{\CM}_{0,m}(X,d) \rightarrow X$ for $1 \leq k\leq m$, 
which send a stable map to the image of the $k$-th marked point in its domain.

For classes 
$\gamma_k \in K_{T}(X)$, $1 \leq k \leq m$, 
the corresponding $m$-point ($T$-equivariant) 
$K$-theoretic Gromov-Witten (KGW) invariant is defined to be
\begin{equation*}
\langle \gamma_1, \gamma_2, \ldots , \gamma_m \rangle_d : = 
\chi^{T}_{\ol{\CM}_{0,m}(X,d)}\left(\prod_{k=1}^m\ev_k^*\gamma_k \right) 
\in K_{T}(\pt).
\end{equation*}
%

%
\subsection{Quantum $K$-theory ring of $G/B$.}
\label{subsec:QKGB}

Let $\sQ := (\sQ_j \mid j \in I)$ be 
the Novikov variables. Following \cite{Giv} and \cite{Lee}, 
the $T$-equivariant (small) quantum $K$-theory ring of $X=G/B$ is defined to be 
\begin{equation*} 
QK_T(X) := K_T(X) \otimes_{K_T(\pt)} K_T(\pt)\bra{\sQ}
\end{equation*}
as a $K_T(\pt)\bra{\sQ}$-module; note that 
$QK_{T}(X)$ is a free module over $K_T(\pt)\bra{\sQ}$, 
with the Schubert classes $\bigl\{ \CO_{w} \mid w \in W \bigr\}$ and  
the opposite Schubert classes $\bigl\{ \CO^{w} \mid w \in W \bigr\}$ as bases.
It is equipped with a commutative and associative 
product (called the quantum product), denoted by $\star$, which is determined by the condition: 
%
%
\begin{equation} \label{eq:qmulti}
\Kmet{\sigma_1 \star \sigma_2}{\sigma_3}=
 \sum_{ d \in Q^{\vee,+} }
 \sQ^d \thpx{\sigma_1}{\sigma_2}{\sigma_3}_{d} \quad 
 \text{for all  } \sigma_1,\sigma_2,\sigma_3\in K_T(G/B), 
\end{equation}
where $\sQ^d := \prod_{j \in I} \sQ_j^{d_j}$ 
for $d = \sum_{j \in I} d_j \alpha_j^{\vee} \in Q^{\vee,+}$, and 
\begin{equation*}
\Kmet{\sigma_1}{\sigma_2} \coloneqq 
\sum_{d \in Q^{\vee,+}} \sQ^d \twpx{\sigma_1}{\sigma_2}_{d}
\end{equation*}
is the quantum $K$-metric. 

In \cite{Kat2}, based on \cite{BF,IMT} (see also \cite{ACT}), 
Kato established an $K_{T}(\pt)$-module isomorphism $\Phi$ from $QK_{T}(X)$ onto 
the $T$-equivariant $K$-group $K_{T}(\QG)$ of the semi-infinite flag manifold $\QG$, 
in which tensor product operation with an arbitrary line bundle class is 
induced from that in $K_{T \times\BC^*}(\QG)$ by the specialization at $q = 1$; 
in our notation, 
the map $\Phi$ sends 
the (opposite) Schubert class $\be^{\mu}\CO^{w} \sQ^{\xi}$ in $QK_{T}(X)$ 
to the corresponding semi-infinite Schubert class $\be^{-\mu}\SQG{ wt_{\xi} }$ 
in $K_{T}(\QG)$ for $\mu \in \Lambda$, $w \in W$, and $\xi \in Q^{\vee,+}$. 
The isomorphism $\Phi$ also respects, in a sense, the quantum product $\star$ 
in $QK_{T}(X)$ and the tensor product $\otimes$ in $K_{T}(\QG)$.
More precisely, one has the commutative diagram:
\begin{equation} \label{eq:qdiagram}
\begin{CD}
QK_{T}(X) @>{\Phi}>{\cong}> K_{T}(\QG) \\
@V{\bullet \,\, \star [\CO_{X}(- \varpi_i)]}VV  
@VV{\bullet \,\, \otimes [\CO_{\QG}(\lng \varpi_i)]}V \\
QK_{T}(X) @>{\Phi}>{\cong}> K_{T}(\QG), 
\end{CD}
\end{equation}
where for $\nu \in \Lambda$, the line bundle $\CO_{X}(-\nu)$ over $X$ denotes 
the $G$-equivariant line bundle constructed as 
the quotient space $G \times^{B} \BC_{\nu}$ of 
the product space $G \times \BC_{\nu}$ by the usual (free) left action of 
the Borel subgroup $B$ of $G$ corresponding to the positive roots, 
with $\BC_{\nu}$ the one-dimensional $B$-module of weight $\nu$; 
we simply write $\CO(- \nu)$ for the class $[\CO_{X}(- \nu)]$ in $K_{T}(X)$ of the line bundle $\CO_{X}(- \nu)$. 
(Here we warn the reader that the conventions of \cite{Kat2} 
differ from those of \cite{KaNS} and this paper,
by the twist coming from the involution $-\lng$.)
Also, we know from \cite{Kat2} that for $w \in W$ and $\xi \in Q^{\vee,+}$, 
$\Phi(\CO^{w} \sQ^{\xi} ) = \st_{\xi} \Phi( \CO^{w} )$ holds, 
and hence that for an arbitrary element $\bullet$ of 
$QK_{T}(X)$ and $\xi \in Q^{\vee,+}$, the equality 
\begin{equation} \label{eq:push}
\Phi(\bullet \, \sQ^{\xi}) = \st_{\xi} \, \Phi(\bullet)
\end{equation}
holds. 

%
\section{Explicit description of certain 3-point KGW invariants.}
\label{sec:main}

A weight $\nu \in P$ is said to be minuscule if 
$\bigl\{ \pair{\nu}{\alpha^{\vee}} \mid \alpha \in \Delta \bigr\} = \bigl\{ -1, 0,1 \bigr\}$; 
if $\nu$ is a dominant minuscule weight, then $\nu = \vpi_{i}$ for some $i \in I$ 
such that $\pair{\varpi_i}{\vartheta^{\vee}} = 1$, 
where $\vartheta^{\vee} \in \Delta^{\vee}$ is the highest coroot. 
Throughout this section, we fix $i \in I$ such that $\vpi_{i}$ is a minuscule fundamental weight, 
and let $\lambda \in W\vpi_{i}$. 

%
\subsection{A special $(-\lambda)$-chain of roots and a reflection order.}
\label{subsec:chain}

Let $x$ be the (unique) minimal-length element in 
$\bigl\{ w \in W \mid w \vpi_{i} = \lambda \bigr\}$, 
and let $x = s_{i_{a}} \cdots s_{i_{1}}$ be a reduced expression for $x$, 
where $a:=\ell(x)$. 
Also, let $y$ be the (unique) element such that 
$yx$ is the (unique) minimal-length element in 
$\bigl\{ w \in W \mid w \vpi_{i} = \lng \vpi_{i} \bigr\}$, 
and let $y = s_{k_{1}} \cdots s_{k_{b}}$ be a reduced expression for $y$, 
where $b:=\ell(y)$. We set 
\begin{align}
& \beta_{c}:=s_{i_{a}} \cdots s_{i_{c+1}}\alpha_{i_{c}} \in \Delta^{+} \quad 
  \text{for $1 \le c \le a$}, \label{eq:beta} \\
& \gamma_{d}:=s_{k_{b}} \cdots s_{k_{d+1}}\alpha_{k_{d}} \in \Delta^{+} \quad 
  \text{for $1 \le d \le b$}. \label{eq:gamma}
\end{align}

The following proposition can be proved in exactly 
the same way as \cite[Lemma 4.1]{LNOS}, in which we assumed that 
$\Fg$ is simply-laced; in the case that $\Fg$ is not simply-laced, 
by making use of the root datum for the (Langlands) dual Lie algebra $\Fg^{\vee}$ of $\Fg$, 
the argument in the simply-laced case goes through.
%
%
\begin{prop} \label{prop:LNOS41}
If we set 
\begin{equation} \label{eq:GammaJ}
\Gamma=(\zeta_{1},\dots,\zeta_{a+b}):=
(\beta_{a},\dots,\beta_{1},-\gamma_{1},\dots,-\gamma_{b}),
\end{equation}
then $\Gamma$ is a reduced $(-\lambda)$-chain from $A_{\circ}$ to $A_{\lambda} = A_{-(-\lambda)}$. 
Moreover, for $1 \le c \le a$, the affine hyperplane between the $(c-1)$-th alcove and 
the $c$-th alcove in $\Gamma$ is $H_{\beta_{a-c+1},0}$, and 
for $1 \le d \le b$, the affine hyperplane between the $(a+d-1)$-th alcove and 
the $(a+d)$-th alcove in $\Gamma$ is $H_{-\gamma_{d},-1} = H_{\gamma_{d},1}${\rm ;}
remark that $\lambda \in H_{\gamma_{d},1}$, and 
hence $s_{\gamma_{d},1}\lambda = \lambda$ for all $1 \le d \le b$. 
\end{prop}

We set 
\begin{align}
& \Delta_{-1}^{+}:=\bigl\{ \beta \in \Delta^{+} \mid 
\pair{\lambda}{\beta^{\vee}} = -1 \bigr\} = \bigl\{ \beta_{a},\dots,\beta_{1} \bigr\}, \label{eq:D-1} \\
& \Delta_{0}^{+}:=\bigl\{ \alpha \in \Delta^{+} \mid 
\pair{\lambda}{\alpha^{\vee}} = 0 \bigr\}, \\
& \Delta_{1}^{+}:=\bigl\{ \gamma \in \Delta^{+} \mid 
\pair{\lambda}{\gamma^{\vee}} = 1 \bigr\} = \bigl\{ \gamma_{1},\dots,\gamma_{b} \bigr\}. \label{eq:D+1}
\end{align}
Let us show the second equality in \eqref{eq:D-1}. 
For the inclusion $\subset$, let $\beta \in \Delta^{+}$ be such that $\pair{\lambda}{\beta^{\vee}} = -1$. 
Then we have $\pair{\vpi_{i}}{x^{-1}\beta^{\vee}} = -1$, which implies that $\beta$ is contained in the 
inversion set $\Inv (x^{-1}) := \Delta^{+} \cap x \Delta^{-}$ for $x^{-1}$. 
Since it is well-known that $\Inv (x^{-1}) = \bigl\{ \beta_{a},\dots,\beta_{1} \bigr\}$, 
we obtain the inclusion $\subset$. For the opposite inclusion $\supset$, 
let $1 \le c \le a$. We have 
\begin{equation*}
\pair{\lambda}{\beta_{c}^{\vee}} = 
\pair{ s_{i_a} \cdots s_{i_1}\vpi_{i} }{ s_{i_{a}} \cdots s_{i_{c+1}}\alpha_{i_{c}}^{\vee} } = 
\pair{ s_{i_c} \cdots s_{i_1}\vpi_{i} }{ \alpha_{i_{c}}^{\vee} }; 
\end{equation*}
notice that $s_{i_c} \cdots s_{i_1}$ is the (unique) minimal-length element in 
$\bigl\{ w \in W \mid w \vpi_{i} = s_{i_c} \cdots s_{i_1} \vpi_{i} \bigr\}$, 
and $(s_{i_c} \cdots s_{i_1})^{-1}\alpha_{i_c} \in \Delta^{-}$ since 
$\ell(s_{i_c} \cdots s_{i_1}) = c$. Hence, it follows that 
$\pair{\lambda}{\beta_{c}^{\vee}} = -1$; see, for example, \cite[Proposition 5.1]{LNSSS1}. 
Thus, we have shown the second equality in \eqref{eq:D-1}. 
Next, let us show the second equality in \eqref{eq:D+1}. We see that 
$y^{-1}$ is the (unique) minimal-length element in 
$\bigl\{ w \in W \mid w (-\lng \vpi_{i}) = - \lambda \bigr\}$. Hence, by the argument above, 
we deduce that $\bigl\{ \gamma \in \Delta^{+} \mid 
\pair{-\lambda}{\gamma^{\vee}} = -1 \bigr\} = \bigl\{ \gamma_{1},\dots,\gamma_{b} \bigr\}$, as desired. 

Also, we obtain the following proposition 
from \cite[Proposition 2.12 and Remark 2.13]{KN}. 
%
%
\begin{prop} \label{prop:ref-chain}
There exists a reflection order $\lhd$ on $\Delta^{+}$ 
satisfying the following conditions{\rm :} 
\begin{enu}
\item $\beta_{a} \lhd \beta_{a-1} \lhd  \cdots \lhd \beta_{2} \lhd \beta_{1}$ 
and $\gamma_{1} \lhd \gamma_{2} \lhd \cdots \lhd \gamma_{b-1} \lhd \gamma_{b}${\rm ;} 
\item $\beta \lhd \alpha \lhd \gamma$ for all $\beta \in \Delta_{-1}^{+}$, 
$\alpha \in \Delta_{0}^{+}$, and $\gamma \in \Delta_{1}^{+}$. 
\end{enu}
\end{prop}

%
\subsection{Chevalley formulas.}
\label{subsec:gc}

Keep the notation and setting of the previous subsection. 
For $w \in W$, we denote by $\bM_{w}$ the set of all directed paths $\bm$ in 
$\QBG(W)$ of the form:
%
%
\begin{equation} \label{eq:def-bMw}
\begin{split}
& \bm:\underbrace{w = x_{0} \edge{ \beta_{a_1} } x_{1} \edge{ \beta_{a_2} } \cdots 
  \edge{ \beta_{a_s} } x_{s}}_{=: \, \bm_{\beta} \ \text{(the $\beta$-part of $\bm$)}} = 
  \underbrace{y_{t} \edge{\gamma_{ b_t }} y_{t-1} \edge{ \gamma_{b_{t-1}} } \cdots 
  \edge{ \gamma_{b_1} } y_{1}}_{=: \, \bm_{\gamma}  \ \text{(the $\gamma$-part of $\bm$)}},  \\
& \text{\rm with $s \ge 0$ and $a \ge a_{1} > a_{2} > \cdots > a_{s} \ge 1$,} \\
& \text{\rm \phantom{with} $t \ge 0$ and $1 \le b_{t} < b_{t-1} < \cdots < b_{1} \le b$}. 
\end{split}
\end{equation}
%
%
\begin{rem} \label{rem:bMw}
Let $\lhd$ be a reflection order satisfying the conditions in Proposition~\ref{prop:ref-chain}. 
An element $\bm \in \bM_{w}$ is a label-increasing directed path starting at $w$ 
(see Theorem~\ref{thm:label} (1)). Hence, by the uniqueness of a label-increasing directed path, 
it follows that if $\bm \ne \bm'$ for $\bm,\,\bm' \in \bM_{w}$, then $\ed(\bm) \ne \ed(\bm')$. 
\end{rem}

For each $j \in I$, let $\st_{j}$ denote the $\BZ[\Lambda]$-linear 
endomorphism $\st_{\alpha_j^{\vee}}$ of $K_{T}(\QG)$ defined in Section~\ref{subsec:equivgroup}, 
and let $\frac{1}{1-\st_{j}}$ denote the infinite sum: 
$1 + \st_{j} + \st_{j}^2 + \cdots$, which is a well-defined $\BZ[\Lambda]$-linear 
endomorphism of $K_{T}(\QG)$. 
%
%
\begin{thm} \label{thm:gc}
Let $w \in W$. Then, the following equalities hold in $K_{T}(\QG)${\rm :} 
\begin{align}
& \LQG{\lng \lambda} \otimes \SQG{w} \nonumber \\[3mm]
& \quad = \sum_{\bm \in \bM_{w}}
  \sum_{ \bchi \in \ol{\Par(-\lambda)} }
  (-1)^{\ell(\bm_{\gamma})} \be^{ - \ed(\bm_{\beta}) \lambda }
  \SQG{ \ed(\bm)t_{\qwt(\bm)+\iota(\bchi)} } \nonumber \\[3mm]
& \quad = \biggl(\prod_{ j \in I,\,
  \pair{\lambda}{\alpha_{j}^{\vee}} = -1 }
  \frac{1}{1-\st_{j}}\biggr) \sum_{\bm \in \bM_{w}}
  (-1)^{\ell(\bm_{\gamma})} \be^{ - \ed(\bm_{\beta}) \lambda }
  \SQG{ \ed(\bm)t_{\qwt(\bm)} }. \label{eq:gchev}
\end{align}
\end{thm}

\begin{proof}
Let us prove the first equality. 
We use the reduced $(-\lambda)$-chain $\Gamma$ given by \eqref{eq:GammaJ}. 
By the definitions, we can naturally identify the set $\CA(w,\Gamma)$ of 
$w$-admissible subsets with the set $\bM_{w}$ of directed paths of the form \eqref{eq:def-bMw}; 
that is, an element $A=(\zeta_{j_1},\dots,\zeta_{j_s}) \in \CA(w,\Gamma)$ for 
$1 \le j_1 < \cdots < j_s \le a+b$ corresponds to the element 
$\bm:w=z_{0} \edge{\zeta_{j_1}} \cdots \edge{\zeta_{j_s}} z_{s}$ of $\bM_{w}$. 
Since $\lambda \in H_{\gamma_{d},1}$ for $1 \le d \le b$, it follows that 
$s_{\gamma_{d},1}\lambda = \lambda$. 
Hence, if $A \in \CA(w,\Gamma)$ corresponds to 
$\bm \in \bM_{w}$ as above, then $\wt(A) = - \ed(\bm_{\beta})\lambda$. 
Also, we see from the definitions that $n(A) = \ell(m_{\gamma})$, $\ed(A) = \ed(\bm)$, and 
$\dnn(A)=\qwt(\bm)$. Substituting these equations into \eqref{eq:gchev-org}, 
we obtain the first equality, as desired. 

For the proof of the second equality, we set 
\begin{equation*}
Q^{\vee,+}_{-1}:=\bigoplus_{
  j \in I,\,
  \pair{\lambda}{\alpha_{j}^{\vee}} = -1 } \BZ_{\ge 0} \alpha_{j}^{\vee}.
\end{equation*}
Since $\lambda$ is a minuscule weight, we see that 
\begin{align*}
& \sum_{\bm \in \bM_{w}}
  \sum_{ \bchi \in \ol{\Par(-\lambda)} }
  (-1)^{\ell(\bm_{\gamma})} \be^{ - \ed(\bm_{\beta}) \lambda }
  \SQG{ \ed(\bm)t_{\qwt(\bm)+\iota(\bchi)} } \\[3mm]
& = \sum_{\bm \in \bM_{w}}
  \sum_{ \xi \in Q^{\vee,+}_{-1} } 
  (-1)^{\ell(\bm_{\gamma})} \be^{ - \ed(\bm_{\beta}) \lambda }
  \SQG{ \ed(\bm)t_{ \qwt(\bm)+\xi } } \\[3mm]
& = \biggl(\prod_{ j \in I,\,
  \pair{\lambda}{\alpha_{j}^{\vee}} = -1 }
  \frac{1}{1-\st_{j}}\biggr) \sum_{\bm \in \bM_{w}}
  (-1)^{\ell(\bm_{\gamma})} \be^{ - \ed(\bm_{\beta}) \lambda }
  \SQG{ \ed(\bm)t_{\qwt(\bm)} }.
\end{align*}
This proves the theorem. 
\end{proof}

Let us consider the special case that $w = e$. 
Let $\bm \in \bM_{e}$, and set $v:=\ed(\bm)$. 
Then, by Proposition~\ref{prop:ref-chain}, 
$\bm$ is a label-increasing directed path from $e$ to $v$ 
with respect to the reflection order $\lhd$; 
in particular, $\bm$ is a shortest directed path from $e$ to $v$. 
Since $e \le v$ in the Bruhat order, all the edges in $\bm$ are 
Bruhat edges, and hence $\qwt(\bm)=0$. Combining this observation with \eqref{eq:gchev}, 
we obtain the following. 
%
%
\begin{cor} \label{cor:gc1}
If $\bm \in \bM_{e}$, then all the edges in $\bm$ are Bruhat edges. 
Hence, the following holds in $K_{T}(\QG)${\rm :} 
\begin{align}
& \LQG{\lng \lambda} = \LQG{\lng \lambda} \otimes \SQG{e} \nonumber \\[3mm]
& \quad = \biggl(\prod_{ j \in I,\,
  \pair{\lambda}{\alpha_{j}^{\vee}} = - 1 }
  \frac{1}{1-\st_{j}}\biggr) \sum_{\bm \in \bM_{e}}
  (-1)^{\ell(\bm_{\gamma})} \be^{ - \ed(\bm_{\beta}) \lambda }
  \SQG{ \ed(\bm) }. \label{eq:gchev1}
\end{align}
\end{cor}
From equation~\eqref{eq:gchev1} in $K_{T}(\QG)$, 
we deduce that, in $K_{T}(X)$ (hence in $QK_{T}(X)$), 
\begin{equation} \label{eq:gchev2}
\CO(- \lambda) = \sum_{\bm \in \bM_{e}}
  (-1)^{\ell(\bm_{\gamma})} \be^{ \ed(\bm_{\beta}) \lambda }
  \CO^{ \ed(\bm) } 
\end{equation}
%
%
by making use of the commutative diagram~\eqref{eq:qdiagram}; 
recall that $\Phi(\be^{\mu}\CO^{w}\sQ^{\xi}) = \be^{-\mu}\SQG{ wt_{\xi} }$ 
for $\mu \in \Lambda$, $w \in W$, $\xi \in Q^{\vee,+}$, and that 
the quantum product $\star$ in $QK_{T}(X)$ becomes 
the ordinary product $\cdot$ in $K_{T}(X)$ under the specialization at $\sQ = 0$. 
Therefore, by Corollary~\ref{cor:gc1}, we obtain 
\begin{equation}\label{eq:minwt}
\Phi \biggl( \biggl(\prod_{ j \in I,\,
  \pair{\lambda}{\alpha_{j}^{\vee}} = - 1 }
\frac{1}{1-\sQ_{j}}\biggr) \CO(- \lambda) \biggr) = \LQG{\lng \lambda}.
\end{equation}
From this equation, again using the commutative diagram~\eqref{eq:qdiagram}, 
we obtain the first half of the following corollary 
by the same proof as that for \cite[Proposition~5.1]{MNS}; recall that $QK_{T}(X)$ is 
generated by the line bundle classes $\CO(- \varpi_i)$, $i \in I$, 
as a $K_{T}(\pt)\bra{\sQ}$-algebra equipped with the quantum product $\star$. 
Also, the second half follows from Theorem~\ref{thm:gc} by using the first half. 
%
%
\begin{cor} \label{cor:gc2}
The following holds in $K_{T}(\QG)${\rm :}
\begin{equation}
\Phi \biggl( \bullet \star 
\biggl(\prod_{ j \in I,\,
  \pair{\lambda}{\alpha_{j}^{\vee}} = - 1 }
\frac{1}{1-\sQ_{j}}\biggr) \CO(- \lambda) \biggr) = 
\Phi(\bullet) \otimes \LQG{\lng \lambda} 
\end{equation}
for an arbitrary element $\bullet$ of $QK_{T}(X)$. 
Hence, the following holds in $QK_{T}(X)${\rm :}
\begin{equation}
\CO(- \lambda) \star \CO^{w} = \sum_{\bm \in \bM_{w}}
  (-1)^{\ell(\bm_{\gamma})} \be^{ \ed(\bm_{\beta}) \lambda }
  \sQ^{\qwt(\bm)} \CO^{ \ed(\bm) } \quad \text{\rm for $w \in W$}. \label{eq:gchev3}
\end{equation}
\end{cor}
From now until the end of this subsection, we assume that $G$ is of type $A_{n}$, i.e., $G = SL_{n+1}(\BC)$, and use the notation of \cite{MNS}. Note that all the fundamental weights $\varpi_{p} = \epsilon_1 + \cdots + \epsilon_p$, $1 \leq p \leq n$, are minuscule, 
where $\epsilon_k := \varpi_k - \varpi_{k-1}$ for $1 \leq k \leq n+1$; we set $\varpi_{0} := 0$ and $\varpi_{n+1} := 0$ by convention. 
Observe that equation~\eqref{eq:minwt} can be rephrased as follows:
\begin{equation}
\Phi(\CO(- \lambda)) = \biggl(\prod_{ j \in I,\,
  \pair{\lambda}{\alpha_{j}^{\vee}} = - 1 } (1-\st_{j})\biggr) \LQG{\lng \lambda}. 
\end{equation}
From this equation, for $0 \leq p \leq k \leq n+1$ and $J \subset [k] := \{1,2,\ldots,k\}$ such that $|J| = p$, we can easily deduce that 
\begin{equation}
\Phi(\CO(- \epsilon_{J})) = 
\left(\prod_{ 
\begin{subarray}{c}
1 \leq j \leq k \\[1mm]
j \notin J, j+1 \in J
\end{subarray}}
 (1-\st_{j})\right)
\LQG{ \lng \eps_{J} }, 
\end{equation}
where $\epsilon_J := \sum_{j \in J} \epsilon_j$ for $J \subset [k]$; 
see also \cite[Eqs.~(2.4) and (2.5)]{MNS2}.
Therefore, for $0 \leq p \leq k \leq n+1$, the image under $\Phi$ of the element
\begin{equation}
\sum_{
\begin{subarray}{c}
J \subset [k] \\[1mm]
|J|=p
\end{subarray}}
\CO(- \epsilon_{J}) \in K_{T}(X) \subset QK_{T}(X)
\end{equation}
is identical to the element 
\begin{equation}
\mathfrak{F}^{k}_{p} := 
 \sum_{
   \begin{subarray}{c}
   J \subset [k] \\[1mm]
   |J|=p
   \end{subarray}}
 \left(\prod_{ 
\begin{subarray}{c}
1 \leq j \leq k \\[1mm]
j \notin J, j+1 \in J
\end{subarray}}
 (1-\st_{j})\right)
\LQG{ \lng \eps_{J} } \in K_{T}(\QG), 
\end{equation}
introduced in \cite{MNS}; 
note that $\mathfrak{F}^{k}_{0}=1$. 
Here, we recall from \cite{MNS} the element 
\begin{equation}
\mathcal{F}^{k}_{p}:=
 \sum_{
   \begin{subarray}{c}
   J \subset [k] \\[1mm]
   |J|=p
   \end{subarray}} \, 
 \left(\prod_{ \begin{subarray}{c} 1 \le j \le k \\[1mm] j,\,j+1 \in J \end{subarray} }
 \frac{1}{1-Q_{j}}\right) 
 \sprod_{j \in J} \CO(-\eps_{j}) \in QK_{T}(X), 
\end{equation}
where $\sprod$ denotes the product with respect to the quantum product $\star$; 
note that $\mathcal{F}^{k}_{0} = 1$. 
Since it is shown in \cite[\S 5]{MNS} (see also \cite[Proposition~3.1]{MNS2}) that $\Phi(\mathcal{F}^k_p) = \mathfrak{F}^k_p$ for $0 \leq p \leq k \leq n+1$, the following equality follows from the injectivity of the map $\Phi$: 
\begin{equation}\label{eq:wedge} 
\mathcal{F}^{k}_{p} = 
\sum_{
\begin{subarray}{c}
J \subset [k] \\[1mm]
|J|=p
\end{subarray}}
\CO(- \epsilon_{J}) \in K_{T}(X) \subset QK_{T}(X) 
\end{equation}
for $0 \leq p \leq k \leq n+1$; 
cf. \cite[Proposition~5.4]{AHKMOX}. 

For the special case $k = n+1$, it is well-known (see, for example, \cite[Introduction]{FL}) that in $K_{T}(X) \subset QK_{T}(X)$, 
\begin{equation}
\sum_{
\begin{subarray}{c}
J \subset [n+1] \\[1mm]
|J|=p
\end{subarray}}
\CO(- \epsilon_{J})
= \sum_{
\begin{subarray}{c}
J \subset [n+1] \\[1mm]
|J|=p
\end{subarray}}
\be^{\epsilon_{J}}, 
\end{equation}
which implies that 
\begin{equation}
\mathcal{F}^{n+1}_{p} = 
\sum_{
\begin{subarray}{c}
J \subset [n+1] \\[1mm]
|J|=p
\end{subarray}}
\be^{\epsilon_{J}} \in K_{T}(X) \subset QK_{T}(X)
\end{equation} 
for $0 \leq p \leq n+1$. 
Thus, we have re-proved \cite[Theorem~5.4]{MNS}. 
%
%
\subsection{Main results.}
\label{subsec:main}

Let $\lambda \in W\vpi_{i}$ be a minuscule weight, and let $x$ and $y$ be as 
in Section~\ref{subsec:chain}. Fix reduced expressions for $x$ and $y$, and 
define $\beta_{c}$, $1 \le c \le a = \ell(x)$, and $\gamma_{d}$, $1 \le d \le b=\ell(y)$, 
as in \eqref{eq:beta} and \eqref{eq:gamma}, respectively. 
Also, fix an arbitrary reflection order $\lhd$ satisfying 
the conditions in Proposition~\ref{prop:ref-chain}. 

Now, let $w,u \in W$, and $d \in Q^{\vee,+}$. We define $\bM_{w}$ as in Section~\ref{subsec:gc}. 
By \eqref{eq:qmulti} and \eqref{eq:gchev3}, we see that 
\begin{align*}
\sum_{\xi \in Q^{\vee,+}}
\sQ^{\xi} \langle \CO(- \lambda),\CO^{w},\CO_{u} \rangle_{\xi} & = 
\sum_{\xi \in Q^{\vee,+}}
\sQ^{\xi} \langle \CO(- \lambda) \star \CO^{w},\CO_{u} \rangle_{\xi} \\
& = 
\sum_{\xi \in Q^{\vee,+}}
\sQ^{\xi} \left \langle 
\sum_{ v \in W,\,\zeta \in Q^{\vee,+} } 
\sS_{\lambda,w}^{v,\zeta} \sQ^{\zeta} \CO^{v},\,\CO_{u} \right\rangle_{\xi} \\
& = \sum_{v \in W,\,\xi,\zeta \in Q^{\vee,+}}
\sS_{\lambda,w}^{v,\zeta} \sQ^{\zeta+\xi} \twp{v}{u}_{\xi},
\end{align*}
where $\sS_{\lambda,w}^{v,\zeta} \in \BZ[P]$ is defined by: 
\begin{equation*}
\CO(- \lambda) \star \CO^{w} = \sum_{v \in W,\,\zeta \in Q^{\vee,+} }
\underbrace{ \sS_{\lambda,w}^{v,\zeta} }_{\in \, \BZ[P]} \sQ^{\zeta} \CO^{v}. 
\end{equation*}
By comparing the coefficients of $\sQ^{d}$ 
on the leftmost-hand side and the rightmost-hand side, we deduce that 
\begin{equation} \label{eq:3pt-lam1}
\langle \CO(- \lambda),\CO^{w},\CO_{u} \rangle_{d} = 
\sum_{v \in W,\, \zeta \in Q^{\vee,+} }
\underbrace{ \sS_{\lambda,w}^{v,\zeta} }_{\in \, \BZ[P]} \twp{v}{u}_{d-\zeta}, 
\end{equation}
where we know from \cite[Lemma~4.1]{LNSX} that 
\begin{align*}
\twp{v}{u}_{d-\zeta} & = 
 \begin{cases}
 1 & \text{if $d-\zeta \ge \qwt(v \Rightarrow u)$}, \\
 0 & \text{otherwise}.
 \end{cases}
\end{align*}
From \eqref{eq:gchev3} and \eqref{eq:3pt-lam1}, we deduce that 
\begin{equation} \label{eq:3pt-lam2}
\langle \CO(- \lambda),\CO^{w},\CO_{u} \rangle_{d} = 
\sum_{ 
  \begin{subarray}{c} 
  \bm \in \bM_{w} \\
  \qwt(\ed(\bm) \Rightarrow u) \le d - \qwt(\bm)
  \end{subarray} } (-1)^{\ell(\bm_{\gamma})} \be^{\ed(\bm_{\beta})\lambda}. 
\end{equation}
Furthermore, we can obtain a cancellations-free formula (in fact, with ``non-negative'' coefficients) 
for $\langle \CO(- \lambda),\CO^{w},\CO_{u} \rangle_{d}$ from equation \eqref{eq:3pt-lam2} as follows. 
For $\bm \in \bM_{w}$, denote by $\bq_{\bm}=\bq_{\bm,u}$ 
the (unique) label-decreasing directed path from $\ed(\bm)$ to $u$ with respect to 
the reflection order $\lhd$ (fixed at the beginning of this subsection); 
note that $\qwt(\ed(\bm) \Rightarrow u) = \qwt(\bq_{\bm})$. 
Denote by $\iota(\bq_{\bm})$ the label of the initial edge of $\bq_{\bm}$; 
if $\bq_{\bm}$ is the trivial directed path of length $0$, 
then we set $\iota(\bq_{\bm}) := \bzero$, where $\bzero$ is a formal symbol 
not contained in $\Delta$; we understand that 
$\bzero \lhd \alpha$ for all $\alpha \in \Delta^{+}$.
We set 
\begin{equation}
\bM_{w,u,d}:=\bigl\{ \bm \in \bM_{w} \mid 
\qwt(\ed(\bm) \Rightarrow u) \le d - \qwt(\bm) \bigr\}, 
\end{equation}
%
%
\begin{equation} \label{eq:def-bfR}
\bfR_{w,u,d}:=\left\{ \bm \in \bM_{w,u,d} \ \Biggm| \ 
\begin{array}{l}
\ell(\bm_{\gamma}) = 0, \, \iota(\bq_{\bm}) \notin \Delta_{1}^{+}, \\[2mm] 
\pair{\vpi_{k}}{\qwt(\ed(\bm) \Rightarrow u)} = \pair{\vpi_{k}}{d - \qwt(\bm)} \\[1mm]
\text{for all $k \in I$ such that $\pair{\lambda}{\alpha_{k}^{\vee}} = 1$}
\end{array} \right\}. 
\end{equation}
%
%
\begin{thm} \label{thm:3pt-lam}
Keep the notation and setting above. We have 
\begin{equation}\label{eq:3ptkgw}
\langle \CO(- \lambda),\CO^{w},\CO_{u} \rangle_{d} = 
\sum_{\bm \in \bfR_{w,u,d}} \be^{\ed(\bm_{\beta})\lambda}. 
\end{equation}
\end{thm}

\begin{proof}
By \eqref{eq:3pt-lam2}, we see that 
\begin{equation}
\langle \CO(- \lambda),\CO^{w},\CO_{u} \rangle_{d} = 
\sum_{ \bm \in \bM_{w,u,d} }(-1)^{\ell(\bm_{\gamma})} \be^{\ed(\bm_{\beta})\lambda}. 
\end{equation}
In order to prove the theorem, it suffices to show that 
\begin{equation}
\sum_{ \bm \in \bB_{w,u,d} }(-1)^{\ell(\bm_{\gamma})} \be^{\ed(\bm_{\beta})\lambda} = 0,
\end{equation}
where 
\begin{align}
\bB_{w,u,d} & := \bM_{w,u,d} \setminus \bfR_{w,u,d} \nonumber \\
& = \left\{ \bm \in \bM_{w,u,d} \ \Biggm| \ 
\begin{array}{l}
\ell( \bm_{\gamma}) > 0, \text{ or } \iota(\bq_{\bm}) \in \Delta_{1}^{+}, \text{ or }\\[2mm] 
\pair{\vpi_{k}}{\qwt(\ed(\bm) \Rightarrow u)} < \pair{\vpi_{k}}{d - \qwt(\bm)} \\[1mm]
\text{for some $k \in I$ such that $\pair{\lambda}{\alpha_{k}^{\vee}} = 1$}
\end{array} \right\}. 
\end{align}
Let $\bm \in \bB_{w,u,d}$, and set $z : = \ed(\bm)$. 
Denote by $\kappa(\bm_{\gamma})$ the label of the final edge of $\bm_{\gamma}$; 
if $\bm_{\gamma}$ is the trivial directed path of length $0$, 
then we set $\kappa(\bm_{\gamma}):=\bzero$.
Set $\psi:=\iota(\bq_{\bm}) \in \Delta^{+} \sqcup \{\bzero\}$ and 
$\vp:=\kappa(\bm_{\gamma})  \in \Delta^{+} \sqcup \{\bzero\}$. 
For $\bm \in \bB_{w,u,d}$, we define 
\begin{equation*}
K_{\bm}:=\left\{ k \in I \ \Biggm| \ 
 \begin{array}{l}
 \alpha_{k} \in \Delta_{1}^{+}, 
 \kappa(\bm_{\gamma})=\vp \lhd \alpha_{k} \rhd \psi = \iota(\bq_{\bm}), \\[1mm]
 \pair{\vpi_{k}}{\qwt(\ed(\bm) \Rightarrow u)} < \pair{\vpi_{k}}{d - \qwt(\bm)}
 \end{array} \right\}. 
\end{equation*}
%
%
\begin{claim} \label{c:3pta}
Assume that $K_{\bm} \ne \emptyset$,  and let $k \in K_{\bm}$. 
If $\bm$ is of the form: 
\begin{equation*}
\underbrace{w \edge{\bullet} \cdots \edge{\bullet} z}_{= \, \bm} = 
\underbrace{\ed(\bm) \edge{\bullet} \cdots \edge{\bullet} u}_{= \, \bq_{\bm}}, 
\end{equation*}
then we define $\Phi(\bm)$ as: 
\begin{equation*}
\underbrace{\overbrace{ w \edge{\bullet} \cdots \edge{\bullet} z}^{= \, \bm} 
\edge{\alpha_{k}} zs_{k}}_{=: \, \Phi(\bm)} = 
\underbrace{\ed(\Phi(\bm)) \edge{\alpha_{k}} 
 \overbrace{z \edge{\bullet} \cdots \edge{\bullet} u}^{= \, \bq_{\bm}} }_{=: \, \bq}. 
\end{equation*}
Then, $\Phi(\bm) \in \bB_{w,u,d}$. 
\end{claim}

\noindent
{\it Proof of Claim~\ref{c:3pta}.}
If the final edge of $\Phi(\bm)$ (labeled by $\alpha_{k}$) is a Bruhat edge, 
then the initial edge of $\bq_{\Phi(\bm)}$ (labeled by $\alpha_{k}$) is a quantum edge. 
Since $\kappa(\bm_{\gamma})=\vp \lhd \alpha_{k}$, 
it follows that $\Phi(\bm)$ is a label-increasing directed path with respect to $\lhd$, 
and hence is a shortest directed path from $w$ to $zs_{k}$. 
Hence, we have $\qwt(\Phi(\bm)) = \qwt(\bm)$. 
Also, by the assumption that $\alpha_{k} \in \Delta_{1}^{+}$, we find that 
$\Phi(\bm) \in \bM_{w}$. 
Similarly, since $\alpha_{k} \rhd \psi = \iota(\bq_{\bm})$, 
it follows that $\bq$ is a label-decreasing directed path with respect to $\lhd$. 
Hence, we have $\bq=\bq_{\Phi(\bm)}$, and 
$\qwt(\ed(\Phi(\bm)) \Rightarrow u) = 
 \qwt(\ed(\bm) \Rightarrow u) + \alpha_{k}^{\vee}$. 
Since $k \in K_{\bm}$ and $\qwt(\ed(\bm) \Rightarrow u) \le d - \qwt(\bm)$, 
we see that $\qwt(\ed(\bm) \Rightarrow u) + \alpha_{k}^{\vee} \le d - \qwt(\bm)$. 
Therefore, we deduce that 
\begin{equation*}
\qwt(\ed(\Phi(\bm)) \Rightarrow u) = 
\qwt(\ed(\bm) \Rightarrow u) + \alpha_{k}^{\vee} \le d - \qwt(\bm)
= d - \qwt(\Phi(\bm)). 
\end{equation*}
Hence, we conclude that $\Phi(\bm) \in \bM_{w,u,d}$; 
it is obvious that $\Phi(\bm) \in \bB_{w,u,d}$ 
since $\ell( \Phi(\bm)_{\gamma}) > 0$ and 
$\iota(\bq_{\Phi(\bm)}) = \alpha_{k} \in \Delta_{1}^{+}$. 

Similarly, if the final edge of $\Phi(\bm)$ (labeled by $\alpha_{k}$) is a quantum edge, 
then the initial edge of $\bq$ (labeled by $\alpha_{k}$) is a Bruhat edge. 
As above, we see that $\Phi(\bm) \in \bM_{w}$, $\bq=\bq_{\Phi(\bm)}$, and 
\begin{equation*}
\begin{split}
& \qwt(\ed(\Phi(\bm)) \Rightarrow u) = 
  \qwt(\ed(\bm) \Rightarrow u), \\
& \qwt(\Phi(\bm)) = \qwt(\bm) + \alpha_{k}^{\vee}. 
\end{split}
\end{equation*}
Also, as seen above, we have 
$\qwt(\ed(\bm) \Rightarrow u) + \alpha_{k}^{\vee} \le d - \qwt(\bm)$, 
and hence $\qwt(\ed(\bm) \Rightarrow u) \le d - ( \qwt(\bm)+\alpha_{k}^{\vee})$. 
Therefore, we deduce that 
\begin{equation*}
\qwt(\ed(\Phi(\bm)) \Rightarrow u) = 
\qwt(\ed(\bm) \Rightarrow u) \le d - (\qwt(\bm)+\alpha_{k}^{\vee}) 
= d - \qwt(\Phi(\bm)). 
\end{equation*}
Hence, we conclude that $\Phi(\bm) \in \bM_{w,u,d}$;
it is obvious that $\Phi(\bm) \in \bB_{w,u,d}$ 
since $\ell( \Phi(\bm)_{\gamma} ) > 0$ and 
$\iota(\bq_{ \Phi(\bm) }) = \alpha_{k} \in \Delta_{1}^{+}$. 
This proves Claim~\ref{c:3pta}. \bqed

\medskip

\begin{claim} \label{c:3ptb}
If $\kappa(\bm_{\gamma}) = \vp = \psi = \iota(\bq_{\bm}) \ne \bzero$, 
then $\vp=\psi$ is a simple root. In this case, 
if $\bm$ is of the form: 
\begin{equation*}
\underbrace{w \edge{\bullet} \cdots \edge{\bullet} zs_{k} \edge{\vp = \alpha_{k}} z}_{= \, \bm} = 
\underbrace{\ed(\bm) \edge{\psi = \alpha_{k}} zs_{k} \edge{\bullet} \cdots \edge{\bullet} u}_{= \, \bq_{\bm}}, 
\end{equation*}
then we define $\Phi(\bm)$ as 
\begin{equation*}
\underbrace{w \edge{\bullet} \cdots \edge{\bullet} zs_{k}}_{=: \, \Phi(\bm)} = 
\underbrace{\ed(\Phi(\bm)) \edge{\bullet} \cdots \edge{\bullet} u}_{=: \, \bq}. 
\end{equation*}
Then, $\Phi(\bm) \in \bB_{w,u,d}$ and $k \in K_{\Phi(\bm)}$. 
\end{claim}

\noindent
{\it Proof of Claim~\ref{c:3ptb}.}
If $y \edge{\beta} ys_{\beta} \edge{\beta} y$ is a directed path in $\QBG(W)$ 
for $y \in W$ and $\beta \in \Delta^{+}$, 
then $\beta$ is a simple root. Indeed, if $y \edge{\beta} ys_{\beta}$ is a Bruhat edge 
(resp., quantum edge), then $ys_{\beta} \edge{\beta} y$ is a quantum edge (resp., Bruhat edge). 
Therefore, by length consideration, we deduce that $2\pair{\rho}{\beta^{\vee}} - 1 = 1$, 
and hence $\pair{\rho}{\beta^{\vee}} = 1$, which implies that $\beta$ is a simple root. 
By using this fact, we find that $\vp=\psi$ is a simple root.

If the final edge of $\bm$ (labeled by $\vp=\alpha_{k}$) is a Bruhat edge, 
then the initial edge of $\bq_{\bm}$ (labeled by $\psi = \alpha_{k}$) is a quantum edge. 
We see that $\Phi(\bm) \in \bM_{w}$, $\bq=\bq_{\Phi(\bm)}$, and that 
\begin{equation*}
\begin{split}
& \qwt(\ed(\Phi(\bm)) \Rightarrow u) = 
  \qwt(\ed(\bm) \Rightarrow u) - \alpha_{k}^{\vee}, \\
& \qwt(\Phi(\bm)) = \qwt(\bm). 
\end{split}
\end{equation*}
Since $\qwt(\ed(\bm) \Rightarrow u) \le d - \qwt(\bm)$ by the assumption, 
it follows that $\qwt(\ed(\Phi(\bm)) \Rightarrow u) \le d - \qwt(\Phi(\bm))$. 
Thus, we have shown that $\Phi(\bm) \in \bM_{w,u,d}$. Also, we see that 
\begin{align*}
& \pair{\vpi_{k}}{\qwt(\ed(\Phi(\bm)) \Rightarrow u)} = 
  \pair{\vpi_{k}}{\qwt(\ed(\bm) \Rightarrow u)} - 1 \\ 
& < \pair{\vpi_{k}}{\qwt(\ed(\bm) \Rightarrow u)} \le \pair{\vpi_{k}}{d - \qwt(\bm)} \\
& = \pair{\vpi_{k}}{d - \qwt(\Phi(\bm))}.
\end{align*}
Hence, we conclude that $\Phi(\bm) \in \bB_{w,u,d}$ and $k \in K_{\Phi(\bm)}$. 

Similarly, if the final edge of $\bm$ (labeled by $\vp=\alpha_{k}$) is 
a quantum edge, then the initial edge of $\bq_{\bm}$ (labeled by $\psi = \alpha_{k}$) 
is a Bruhat edge. We see that $\Phi(\bm) \in \bM_{w}$, $\bq=\bq_{\Phi(\bm)}$, and that 
\begin{equation*}
\begin{split}
& \qwt(\ed(\Phi(\bm)) \Rightarrow u) = 
  \qwt(\ed(\bm) \Rightarrow u), \\
& \qwt(\Phi(\bm)) = \qwt(\bm) - \alpha_{k}^{\vee}. 
\end{split}
\end{equation*}
Since $\qwt(\ed(\bm) \Rightarrow u) \le d - \qwt(\bm)$ by the assumption, 
it follows that $\qwt(\ed(\Phi(\bm)) \Rightarrow u) \le d - \qwt(\Phi(\bm))$. 
Thus, we have shown that $\Phi(\bm) \in \bM_{w,u,d}$, as desired. 
Also, we see that 
\begin{align*}
& \pair{\vpi_{k}}{\qwt(\ed(\Phi(\bm)) \Rightarrow u)} = 
  \pair{\vpi_{k}}{\qwt(\ed(\bm) \Rightarrow u)} \\
& \le \pair{\vpi_{k}}{d - \qwt(\bm)} 
  = \pair{\vpi_{k}}{d - \qwt(\Phi(\bm))} -1 \\
& < \pair{\vpi_{k}}{d - \qwt(\Phi(\bm))}. 
\end{align*}
Hence, we conclude that $\Phi(\bm) \in \bB_{w,u,d}$ and $k \in K_{\Phi(\bm)}$. 
This proves Claim~\ref{c:3ptb}. \bqed

\medskip

Now, we set
\begin{equation}
\bC_{w,u,d}:=\bigl\{
 \bm \in \bB_{w,u,d} \mid 
 \kappa(\bm_{\gamma}) \ne \iota(\bq_{\bm}) \text{ or } 
 \kappa(\bm_{\gamma}) = \iota(\bq_{\bm}) = \bzero \bigr\}.
\end{equation}
For each $\bm \in \bC_{w,u,d}$, let $\bK_{\bm}$ be the set of 
all sequences $\bk = (k_{1},\dots,k_{t})$ of elements in $K_{\bm}$, 
where $0 \le t \le \# K_{\bm}$, 
such that $\alpha_{k_{1}} \lhd \cdots \lhd \alpha_{k_{t}}$. 
For $\bk = (k_{1},\dots,k_{t}) \in \bK_{\bm}$, we define $\bm \ast \bk$ as: 
\begin{align*}
& \underbrace{\overbrace{ w \edge{\bullet} \cdots \edge{\bullet} z }^{= \, \bm} 
\edge{\alpha_{k_1}} \bullet \edge{\alpha_{k_2}} \cdots 
\edge{\alpha_{k_{t-1}}} \bullet \edge{\alpha_{k_t}} zs_{k_1}s_{k_2} \cdots s_{k_t} }_{=: \, \bm \ast \bk} \\
& = \underbrace{\ed(\bm \ast \bk) \edge{\alpha_{k_t}} \bullet \edge{\alpha_{k_{t-1}}} \cdots 
\edge{\alpha_{k_2}} \bullet \edge{\alpha_{k_1}} 
\overbrace{z \edge{\bullet} \cdots \edge{\bullet} u}^{\bq_{\bm}}}_{= \, \bq_{\bm \ast \bk}}; 
\end{align*}
remark that $\bm \ast \bk$ is a label-increasing directed path. 
By Claims~\ref{c:3pta} and \ref{c:3ptb}, we find that
\begin{equation}
\bB_{w,u,d} = \bigsqcup_{ \bm \in \bC_{w,u,d} }
\bigl\{ \bm \ast \bk \mid \bk \in \bK_{\bm} \bigr\}. 
\end{equation}
Also, it is easily verified that for $\bm \in \bC_{w,u,d}$ 
with $K_{\bm} \ne \emptyset$, 
\begin{equation}
\sum_{ \bk \in \bK_{\bm} }(-1)^{\ell( (\bm \ast \bk)_{\gamma})} \be^{\ed( (\bm \ast \bk)_{\beta} )\lambda} =
(-1)^{\ell(\bm_{\gamma})} \be^{ \ed(\bm_{\beta})\lambda } 
\sum_{ \bk \in \bK_{\bm} }(-1)^{\ell(\bk)} = 0.
\end{equation}
Therefore, we deduce that 
\begin{equation} \label{eq:3ptx}
\sum_{ \bm \in \bB_{w,u,d} }(-1)^{\ell(\bm_{\gamma})} \be^{\ed(\bm_{\beta})\lambda} = 
\sum_{ \begin{subarray}{c} \bm \in \bC_{w,u,d} \\ K_{\bm}=\emptyset \end{subarray} }
(-1)^{\ell(\bm_{\gamma})} \be^{\ed(\bm_{\beta})\lambda}.
\end{equation}
In order to show that the right-hand side of \eqref{eq:3ptx} is equal to $0$, 
we will define a sijection (i.e., a sign-reversing bijection) $\Psi$ on the set
\begin{equation}
\bC_{w,u,d}^{\emptyset}:=
\bigl\{ \bm \in \bC_{w,u,d} \mid K_{\bm} = \emptyset \bigr\}. 
\end{equation}
Let $\bm \in \bC_{w,u,d}^{\emptyset}$; recall that 
$\vp=\kappa(\bm_{\gamma}) \in \Delta_{1}^{+} \sqcup \{\bzero\}$, 
$\psi=\iota(\bq_{\bm}) \in \Delta^{+} \sqcup \{\bzero\}$, and $z=\ed(\bm)$. 
We claim that $\vp \ne \bzero$ (i.e., $\ell(\bm_{\gamma}) > 0$) or 
$\psi \in \Delta_{1}^{+}$. Suppose, for a contradiction, that 
$\vp = \bzero$ and $\psi \notin \Delta_{1}^{+}$. By the definition of 
$\bB_{w,u,d}$, there exists $k \in I$ such that 
$\alpha_{k} \in \Delta_{1}^{+}$ and 
$\pair{\vpi_{k}}{\qwt(\ed(\bm) \Rightarrow u)} < 
\pair{\vpi_{k}}{d - \qwt(\bm)}$. Since $\psi = \bzero$ or 
$\psi \in \Delta_{-1}^{+} \sqcup \Delta_{0}^{+}$, we deduce that 
$k \in K_{\bm}$, which contradicts the assumption that $\bm \in \bC_{w,u,d}^{\emptyset}$. 

\medskip

\paragraph{\bf Case 1.}
%
Assume that $\kappa(\bm_{\gamma}) = \vp \rhd \psi = \iota(\bq_{\bm})$; 
note that $\vp \ne \bzero$ in this case.  
We define $\Psi(\bm)$ to be the directed path obtained from $\bm$ by removing the final edge.
Notice that $\Psi(\bm) \in \bM_{w}$ with $\ed(\Psi(\bm))=zs_{\vp}$, and that 
$\bq_{\Psi(\bm)}$ is identical to the directed path obtained 
from $\bq_{\bm}$ by adding an edge labeled by $\vp$ at the beginning of $\bq_{\bm}$; 
that is, 
\begin{equation*}
\underbrace{w \edge{\bullet} \cdots \edge{\bullet} zs_{\vp} \edge{\vp} z}_{= \, \bm} = 
\underbrace{\ed(\bm) \edge{\psi} \cdots \edge{\bullet} u}_{= \, \bq_{\bm}}, 
\end{equation*}
and
\begin{equation*}
\underbrace{w \edge{\bullet} \cdots \edge{\bullet} zs_{\vp}}_{= \, \Psi(\bm)} = 
\underbrace{\ed(\Psi(\bm)) \edge{\vp} z = \ed(\bm) \edge{\psi} \cdots \edge{\bullet} u}_{= \, 
\bq_{\Psi(\bm)}}. 
\end{equation*}
We claim that $\Psi(\bm) \in \bM_{w,u,d}$. 
If the final edge of $\bm$ (labeled by $\vp$) is a Bruhat edge, then we have
\begin{equation} \label{eq:case1a}
\begin{split}
& \qwt(\ed(\Psi(\bm)) \Rightarrow u) = \qwt(\ed(\bm) \Rightarrow u), \\
& \qwt(\Psi(\bm)) = \qwt(\bm).
\end{split}
\end{equation}
Since $\qwt(\ed(\bm) \Rightarrow u) \le d - \qwt(\bm)$ by the assumption, 
it follows that $\qwt(\ed(\Psi(\bm)) \Rightarrow u) \le d - \qwt(\Psi(\bm))$. 
If the final edge of $\bm$ (labeled by $\vp$) is 
a quantum edge, then we have
\begin{equation} \label{eq:case1b}
\begin{split}
& \qwt(\ed(\Psi(\bm)) \Rightarrow u) = 
  \qwt(\ed(\bm) \Rightarrow u) + \vp^{\vee}, \\
& \qwt(\Psi(\bm)) = \qwt(\bm) - \vp^{\vee}.
\end{split}
\end{equation}
Since $\qwt(\ed(\bm) \Rightarrow u) \le d - \qwt(\bm)$ by the assumption, 
it follows that $\qwt(\ed(\Psi(\bm)) \Rightarrow u) \le d - \qwt(\Psi(\bm))$. 
Thus, we have shown that $\Psi(\bm) \in \bM_{w,u,d}$, as desired. 
Also, we see that $\Psi(\bm) \in \bB_{w,u,d}$ 
since $\iota(\bq_{\Psi(\bm)}) = \vp \in \Delta_{1}^{+}$. 
Since $\kappa( \Psi(\bm)_{\gamma} ) \lhd \vp = \iota(\bq_{\Psi(\bm)})$, 
we deduce that $\Psi(\bm) \in \bC_{w,u,d}$. 
Suppose, for a contradiction, that $K_{\Psi(\bm)} \ne \emptyset$. 
Let $k \in K_{\Psi(\bm)}$; note that $\alpha_{k} \in \Delta_{1}^{+}$, 
$\alpha_{k} \rhd \vp \rhd \psi$, and that 
\begin{equation*}
\pair{\vpi_{k}}{\qwt(\ed(\Psi(\bm)) \Rightarrow u)} < \pair{\vpi_{k}}{d- \qwt(\Psi(\bm))}.
\end{equation*}
We have $\kappa(\bm_{\gamma}) = \vp \lhd \alpha_{k} \rhd \psi = \iota(\bq_{\bm})$. 
Also, we deduce from \eqref{eq:case1a} and \eqref{eq:case1b} that 
\begin{equation*}
\pair{\vpi_{k}}{\qwt(\ed(\bm) \Rightarrow u)} < \pair{\vpi_{k}}{d- \qwt(\bm)}.
\end{equation*}
Hence, we obtain $k \in K_{\bm}$, which contradicts $\bm \in \bC_{w,u,d}^{\emptyset}$. 
Thus, we conclude that $\Psi(\bm) \in \bC_{w,u,d}^{\emptyset}$; notice that 
$\Psi(\bm)$ satisfies the condition of Case 2 below.

\medskip

\paragraph{\bf Case 2.}
%
Assume that either of the following holds: 
(i) $\vp \ne \bzero$ and $\vp \lhd \psi$ (note that $\psi \in \Delta_{1}^{+}$ 
in this case); (ii) $\vp = \bzero$ and $\psi \in \Delta_{1}^{+}$. Namely, we have 
\begin{equation*}
\text{(i)} \quad
\underbrace{w \edge{\bullet} \cdots \edge{\bullet} zs_{\vp} \edge{\vp} z}_{= \, \bm} = 
\underbrace{\ed(\bm) \edge{\psi} \cdots \edge{\bullet} u}_{= \, \bq_{\bm}}, \quad \text{or} 
\end{equation*}
\begin{equation*}
\text{(ii)} \quad 
\underbrace{w \edge{\bullet} \cdots \edge{\bullet} z}_{= \, \bm \, = \, \bm_{\beta}} = 
\underbrace{\ed(\bm) \edge{\psi} \cdots \edge{\bullet} u}_{= \, \bq_{\bm}}.  
\end{equation*}
Then we define $\Psi(\bm)$ as: 
\begin{equation*}
\text{(i)} \quad
\underbrace{w \edge{\bullet} \cdots \edge{\bullet} zs_{\vp} \edge{\vp} z = \ed(\bm) 
\edge{\psi} zs_{\psi}}_{=: \, \Psi(\bm)} = 
\underbrace{\ed(\Psi(\bm)) \edge{\bullet} \cdots \edge{\bullet} u}_{= \, \bq_{\Psi(\bm)}},  
\end{equation*}
\begin{equation*}
\text{(ii)} \quad 
\underbrace{\overbrace{ w \edge{\bullet} \cdots \edge{\bullet} z}^{= \, \bm \, = \, \bm_{\beta}} = \ed(\bm) 
\edge{\psi} zs_{\psi}}_{=: \, \Psi(\bm)} = 
\underbrace{\ed(\Psi(\bm)) \edge{\bullet} \cdots \edge{\bullet} u}_{= \, \bq_{\Psi(\bm)}}. 
\end{equation*}
We claim that $\Psi(\bm) \in \bM_{w,u,d}$. 
If the initial edge of $\bq_{\bm}$ (labeled by $\psi$) is a Bruhat edge, then we have
\begin{equation} \label{eq:case2a}
\begin{split}
& \qwt(\ed(\Psi(\bm)) \Rightarrow u) = \qwt(\ed(\bm) \Rightarrow u), \\
& \qwt(\Psi(\bm)) = \qwt(\bm).
\end{split}
\end{equation}
Since $\qwt(\ed(\bm) \Rightarrow u) \le d - \qwt(\bm)$ by the assumption, 
it follows that $\qwt(\ed(\Psi(\bm)) \Rightarrow u) \le d - \qwt(\Psi(\bm))$. 
If the initial edge of $\bq_{\bm}$ (labeled by $\psi$) is 
a quantum edge, then we have
\begin{equation} \label{eq:case2b}
\begin{split}
& \qwt(\ed(\Psi(\bm)) \Rightarrow u) = 
  \qwt(\ed(\bm) \Rightarrow u) - \psi^{\vee}, \\
& \qwt(\Psi(\bm)) = \qwt(\bm) + \psi^{\vee}.
\end{split}
\end{equation}
Since $\qwt(\ed(\bm) \Rightarrow u) \le d - \qwt(\bm)$ by the assumption, 
it follows that $\qwt(\ed(\Psi(\bm)) \Rightarrow u) \le d - \qwt(\Psi(\bm))$. 
Thus, we have shown that $\Psi(\bm) \in \bM_{w,u,d}$, as desired. 
Also, we see that $\Psi(\bm) \in \bB_{w,u,d}$ 
since $\ell( \Psi(\bm)_{\gamma} ) > 0$. 
Since $\kappa( \Psi(\bm)_{\gamma} ) = \psi \rhd \iota(\bq_{\Psi(\bm)})$, 
we deduce that $\Psi(\bm) \in \bC_{w,u,d}$. 
Suppose, for a contradiction, that $K_{\Psi(\bm)} \ne \emptyset$. 
Let $k \in K_{\Psi(\bm)}$; note that $\alpha_{k} \in \Delta_{1}^{+}$, 
$\vp \lhd \psi \lhd \alpha_{k}$, and 
\begin{equation*}
\pair{\vpi_{k}}{\qwt(\ed(\Psi(\bm)) \Rightarrow u)} < \pair{\vpi_{k}}{d- \qwt(\Psi(\bm))}.
\end{equation*}
We have $\kappa(\bm_{\gamma}) = \vp \lhd \alpha_{k} \rhd \psi = \iota(\bq_{\bm})$. 
Also, we deduce from \eqref{eq:case2a} and \eqref{eq:case2b} that 
\begin{equation*}
\pair{\vpi_{k}}{\qwt(\ed(\bm) \Rightarrow u)} < \pair{\vpi_{k}}{d- \qwt(\bm)}.
\end{equation*}
Hence, we obtain $k \in K_{\bm}$, which contradicts $\bm \in \bC_{w,u,d}^{\emptyset}$. 
Thus, we conclude that $\Psi(\bm) \in \bC_{w,u,d}^{\emptyset}$; notice that 
$\Psi(\bm)$ satisfies the condition of Case 1 above. 

\medskip

In this way, we have defined a sijection $\bm \mapsto \Psi(\bm)$ on $\bC_{w,u,d}^{\emptyset}$. 
Therefore, we conclude that 
\begin{equation*}
\sum_{ \bm \in \bC_{w,u,d}^{\emptyset}  }
(-1)^{\ell(\bm_{\gamma})} \be^{\ed(\bm_{\beta})\lambda} = 0, 
\end{equation*}
as desired. This completes the proof of Theorem~\ref{thm:3pt-lam}. 
\end{proof}

In Appendix~\ref{sec:example}, we give a few examples of Theorem~\ref{thm:3pt-lam}. 

\begin{rem}
It would be interesting to find a geometric explanation of the ``non-negativity'' of 
the coefficients on the right-hand side of equation~\eqref{eq:3ptkgw}. 
\end{rem}

Now, we define elements $\sC_{\lambda,w}^{v} \in \BZ[\Lambda]$ by: 
\begin{equation}
\CO(- \lambda) \cdot \CO^{w} = 
\sum_{v \in W} \sC_{\lambda,w}^{v} \CO^{v},
\end{equation}
where $\cdot$ is the ordinary product in $K_{T}(G/B)$. 
Recall that 
\begin{equation*}
\CO(- \lambda) \star \CO^{w} = \sum_{v \in W,\,\zeta \in Q^{\vee,+} }
\sS_{\lambda,w}^{v,\zeta} \sQ^{\zeta} \CO^{v} = 
\sum_{\bm \in \bM_{w}}
  (-1)^{\ell(\bm_{\gamma})} \be^{ \ed(\bm_{\beta}) \lambda }
  \sQ^{\qwt(\bm)} \CO^{ \ed(\bm) }. 
\end{equation*}
From these equalities, we deduce that 
\begin{equation*}
\sC_{\lambda,w}^{v} = \sum_{\zeta \in Q^{\vee,+}}
\sS_{\lambda,w}^{v,\zeta} \sQ^{\zeta} \biggm|_{\sQ=0} = 
\sum_{ \begin{subarray}{c} \bm \in \bM_{w} \\[1mm] \ed(\bm) = v \end{subarray} }
  (-1)^{\ell(\bm_{\gamma})} \be^{ \ed(\bm_{\beta}) \lambda }
  \sQ^{\qwt(\bm)} \biggm|_{\sQ=0}. 
\end{equation*}
We set 
\begin{equation}
U:=\bigl\{ v \in W \mid 
 \text{$\ed(\bm) = v$ and $\qwt(\bm)=0$ for some $\bm \in \bM_{w}$} \bigr\}.
\end{equation}
Then we see that 
\begin{enumerate}
\item[(a)] if $v \in U$, then 
\begin{equation} \label{eq:ca}
\sum_{\zeta \in Q^{\vee,+}}
\sS_{\lambda,w}^{v,\zeta} \sQ^{\zeta}  \in 
\sC_{\lambda,w}^{v} + \sum_{ 
 \begin{subarray}{c} j \in I \\ \pair{\lambda}{\alpha_{j}^{\vee}} \ne 0 \end{subarray} }
 \sQ_{j} \BZ[\sQ], 
\end{equation}

\item[(b)] if $v \in W \setminus U$, then 
\begin{equation} \label{eq:cb}
\sum_{\zeta \in Q^{\vee,+}}
\sS_{\lambda,w}^{v,\zeta} \sQ^{\zeta}  \in \sum_{ 
 \begin{subarray}{c} j \in I \\ \pair{\lambda}{\alpha_{j}^{\vee}} \ne 0 \end{subarray} }
 \sQ_{j} \BZ[\sQ];
\end{equation}
in particular, we conclude that $\sC_{\lambda,w}^{v} = 0$. 
\end{enumerate}
%
%
\begin{thm} \label{thm:3pt-lam2}
Keep the notation and setting above. We write $d \in Q^{\vee,+}$ as 
$d = \sum_{j \in I} d_{j}\alpha_{j}^{\vee}$, with $d_{j} \in \BZ_{\ge 0}$ for $j \in I$. 
If $d_{j}=0$ for all $j \in I$ such that $\pair{\lambda}{\alpha_{j}^{\vee}} \ne 0$ 
{\rm (}or equivalently, if $\pair{\lambda}{d}=0${\rm )}, then the following holds{\rm :} 
\begin{equation}
\langle \CO(- \lambda),\CO^{w},\CO_{u} \rangle_{d} = 
\langle \CO(- \lambda) \cdot \CO^{w},\CO_{u} \rangle_{d}. 
\end{equation}
\end{thm}

\begin{proof}
Recall that 
\begin{align*}
& \sum_{\xi \in Q^{\vee,+}}
\sQ^{\xi} \langle \CO(- \lambda),\CO^{w},\CO_{u} \rangle_{\xi} 
= \sum_{v \in W} \left( \sum_{\zeta \in Q^{\vee,+}}\sS_{\lambda,w}^{v,\zeta} \sQ^{\zeta} \right)
\sum_{\xi \in Q^{\vee,+}} \sQ^{\xi} \twp{v}{u}_{\xi} \\[3mm]
& 
= \sum_{v \in U} \biggl( 
\underbrace{ \sum_{\zeta \in Q^{\vee,+}}
\sS_{\lambda,w}^{v,\zeta} \sQ^{\zeta}}_{ \text{see \eqref{eq:ca}} } \biggr)
\sum_{\xi \in Q^{\vee,+}} \sQ^{\xi} \twp{v}{u}_{\xi} + 
\sum_{v \in W \setminus U} \biggl( 
\underbrace{ \sum_{\zeta \in Q^{\vee,+}}
\sS_{\lambda,w}^{v,\zeta} \sQ^{\zeta}}_{ \text{see \eqref{eq:cb}} } \biggr)
\sum_{\xi \in Q^{\vee,+}} \sQ^{\xi} \twp{v}{u}_{\xi}. 
\end{align*}
By comparing the coefficients of $\sQ^{d}$ 
on the leftmost-hand side and the rightmost-hand side, we deduce that 
\begin{align*}
\langle \CO(- \lambda),\CO^{w},\CO_{u} \rangle_{d}
& = \sum_{v \in U} \biggl( \sum_{\zeta \in Q^{\vee,+}} 
\sS_{\lambda,w}^{v,\zeta} \sQ^{\zeta} \biggm|_{\sQ=0} \biggr)
\twp{v}{u}_{d} \\[2mm]
& = \sum_{v \in U} \sC_{\lambda,w}^{v} \twp{v}{u}_{d}
  = \sum_{v \in W} \sC_{\lambda,w}^{v} \twp{v}{u}_{d} \\
& = \langle \CO(- \lambda) \cdot \CO^{w},\CO_{u} \rangle_{d}. 
\end{align*}
This proves the theorem. 
\end{proof}

\appendix

\section{Examples.}
\label{sec:example}

In this appendix, 
for an edge $x \edge{\alpha} y$ in $\QBG(W)$, 
we write $x \Be{\alpha} y$ (resp., $x \Qe{\alpha} y$) 
to indicate that the edge is a Bruhat (resp., quantum) edge. 

We assume that $\Fg$ is of type $A_{3}$. 
We set $\alpha_{ij}:=\sum_{k=i}^{j}\alpha_{k} \in \Delta^{+}$ for $1 \le i \le j \le 3$, and 
$\eps_{2}:=\vpi_{2}-\vpi_{1}=s_{1}\vpi_{1} \in W\vpi_{1}$ (note that $\vpi_{1}$ is a minuscule fundamental weight). 
In the notation of Section~\ref{subsec:chain}, we have $x = s_{1}$ with $a = 1$, and 
$y = s_{3}s_{2}$ with $b = 2$; also, we have 
\begin{equation*}
\Gamma=(-\beta_{1},\gamma_{1},\gamma_{2}) = ( - \alpha_{1}, \alpha_{23}, \alpha_{2}). 
\end{equation*}

%
Let $w = s_{2}$. Then, the set $\bM_{s_{2}}$ consists of the following elements:
\begin{align*}
& \bm_{0} : s_{2} \quad \text{(trivial directed path of length $0$)}, & 
& \bm_{1} : s_{2} \Be{\alpha_{1}} s_{2}s_{1}, \\
& s_{2} \Be{\alpha_{1}} s_{2}s_{1} \Be{\alpha_{2}} s_{2}s_{1}s_{2}, & 
& s_{2} \Be{\alpha_{23}} s_{3}s_{2}, \\
& s_{2} \Be{\alpha_{23}} s_{3}s_{2} \Qe{\alpha_{2}} s_{3}, & 
& s_{2} \Qe{\alpha_{2}} e. 
\end{align*}
By \eqref{eq:gchev3}, we have
\begin{align*}
\CO(- \eps_{2}) \star \CO^{s_{2}} & = 
\be^{s_{2} \eps_{2} } \CO^{s_{2}} + 
\be^{ s_{2}s_{1} \eps_{2} } \CO^{s_{2}s_{1}} -
\be^{ s_{2}s_{1} \eps_{2} } \CO^{s_{2}s_{1}s_{2}} -
\be^{ s_{2} \eps_{2} } \CO^{s_{3}s_{2}} \\
& \quad + 
\be^{ s_{2} \eps_{2} } \sQ_{2} \CO^{s_{3}} - 
\be^{ s_{2} \eps_{2} } \sQ_{2} \CO^{e}.
\end{align*}
Also, we see that for all $u \in W$ and $d \in Q^{\vee,+}$,
\begin{equation*}
\bfR_{s_{2},u,d} \subset \bigl\{ s_{2}, \ s_{2} \Be{\alpha_1} s_{2}s_{1} \bigr\} = 
\bigl\{ \bm_{0},\, \bm_{1} \bigr\};
\end{equation*}
note that $\qwt(\bm) = 0$ for all $\bm \in \bfR_{s_{2},u,d}$. 
Let $d=d_{1}\alpha_{1}^{\vee} + d_{2}\alpha_{2}^{\vee} + d_{3}\alpha_{3}^{\vee} \in Q^{\vee,+}$. 
If $d_{2} \ge 2$, then we find that $\bfR_{s_{2},u,d} = \emptyset$ for all $u \in W$. 
Indeed, if $u \in W$ and $\bm \in \bfR_{s_{2},u,d}$, then we would have 
\begin{equation*}
\qwt(\ed(\bm) \Rightarrow u) \le \qwt(\ed(\bm) \Rightarrow e) + \underbrace{\qwt(e \Rightarrow u)}_{= \, 0}
\le \alpha_{1}^{\vee} + \alpha_{2}^{\vee}. 
\end{equation*}
Hence, 
$\pair{\vpi_{2}}{\qwt(\ed(\bm) \Rightarrow u)} \le 1 < d_{2} 
= \pair{\vpi_{2}}{d - \qwt(\bm)}$, which is a contradiction. 
Thus, we deduce that $\bfR_{s_{2},u,d} = \emptyset$, 
and hence by Theorem~\ref{thm:3pt-lam}, 
\begin{equation*}
\langle \CO(- \eps_{2}),\CO^{s_{2}},\CO_{u} \rangle_{d} = 0 \quad 
\text{for all $u \in W$ (if $d_{2} \ge 2$)}. 
\end{equation*}

Assume next that $d_{2} = 1$. If $u \ge s_{2}s_{1}$ in the Bruhat order $\ge$ on $W$, then 
$\qwt(\ed(\bm) \Rightarrow u) = 0$. By the same argument as above, 
we deduce that $\bfR_{s_{2},u,d} = \emptyset$ in this case. 
Here, observe that there are 12 elements $u \in W$
which do not satisfy $u \ge s_{2}s_{1}$:
\begin{equation*}
W_{\not\ge s_2s_1}:=
\bigl\{ e, s_1, s_2, s_1s_2, s_3, s_3s_1, s_3s_2, s_3s_1s_2, 
s_2 s_3, s_2s_3s_2, s_1s_2s_3, s_1s_2s_3s_2 \bigr\}.
\end{equation*}
Let us check if $\bm_{0} \in \bfR_{s_{2},u,d}$ for each $u \in W_{\not \ge s_2s_1}$. 
If $u \ge s_{2}$, then we see by the same argument as above that $\bm_{0} \notin \bfR_{s_{2},u,d}$. 
Also, for all $u \in W_{\not\ge s_2}:=\bigl\{ e, s_1, s_3, s_3s_1 \bigr\}$, 
we have $\iota(\bq_{\bm_{0}}) = \alpha_{2} \in \Delta_{1}^{+}$. Thus, we deduce that 
$\bm_{0} \notin \bfR_{s_{2},u,d}$ for all $u \in W_{\not\ge s_2}$, 
and hence for all $u \in W_{\not\ge s_2s_1}$. 
Let us check if $\bm_{1} \in \bfR_{s_{2},u,d}$ for each $u \in W_{\not \ge s_2s_1}$. 
If $u \ge s_{2}$, then we have
\begin{equation*}
\qwt(\ed(\bm_{1}) \Rightarrow u) \le \qwt(s_2s_1 \Rightarrow s_2) + 
\underbrace{\qwt(s_2 \Rightarrow u)}_{= \, 0} = \alpha_{1}^{\vee}, 
\end{equation*}
which implies that 
$\pair{\vpi_{2}}{\qwt(\ed(\bm_{1}) \Rightarrow u)} = 0$. 
Hence, $\bm_{1} \not\in \bfR_{s_{2},u,d}$. 
Also, for all $u \in W_{\not\ge s_2}$, we have 
$\iota(\bq_{\bm_{1}}) = \alpha_{2} \in \Delta_{1}^{+}$. Thus, we see that 
$\bm_{1} \notin \bfR_{s_{2},u,d}$ for all $u \in W_{\not\ge s_2}$, and hence 
for all $u \in W_{\not\ge s_2s_1}$. 
Therefore, we conclude that $\bfR_{s_{2},u,d} = \emptyset$ 
for all $u \in W_{\not\ge s_2s_1}$. Combining these, 
we deduce from Theorem~\ref{thm:3pt-lam} that 
\begin{equation*}
\langle \CO(- \eps_{2}),\CO^{s_{2}},\CO_{u} \rangle_{d} = 0 \quad 
\text{for all $u \in W$ (if $d_{2} = 1$)}. 
\end{equation*}

Assume finally that $d_{2} = 0$. For example, let us consider the case 
that $d_{1} > 0$ and $u = s_{2},\,s_{2}s_{1}$. In this case, we see that 
$\bfR_{s_{2},u,d} = \bigl\{ \bm_{0}, \bm_{1} \bigr\}$. Hence, it follows from Theorem~\ref{thm:3pt-lam} that 
\begin{equation*}
\langle \CO(- \eps_{2}),\CO^{s_{1}},\CO_{u} \rangle_{d} = 
\be^{s_{2}\eps_{2}} + \be^{s_{2}s_{1}\eps_{2}} = 
\be^{s_{2}s_{1}\vpi_{1}} + \be^{\vpi_{1}}. 
\end{equation*}
Also, if $d_{1} > 0$ and $u = s_{3}s_{2},\,s_{3}s_{2}s_{1}$, then $\bfR_{s_{2},u,d} = 
\bigl\{ \bm_{1} \bigr\}$. Hence, it follows from Theorem~\ref{thm:3pt-lam} that 
\begin{equation*}
\langle \CO(- \eps_{2}),\CO^{s_{1}},\CO_{u} \rangle_{d} = \be^{s_{2}s_{1}\eps_{2}} 
= \be^{\vpi_{1}}.
\end{equation*}

\end{document}